\documentclass{amsart}
\usepackage{amsmath, amssymb}
\input xy
\xyoption{all}
\begin{document}

% These will be typeset in italics
\newtheorem{theorem}{Theorem}[section]
\newtheorem{proposition}[theorem]{Proposition}
\newtheorem{lemma}[theorem]{Lemma}
\newtheorem{corollary}[theorem]{Corollary}
\newtheorem{fact}[theorem]{Fact}

% These will be typeset in Roman
\theoremstyle{definition}
\newtheorem{definition}[theorem]{Definition}
\newtheorem{conjecture}[theorem]{Conjecture}
\newtheorem{claim}[theorem]{Claim}
\newtheorem{reduction}[theorem]{Reduction}

\theoremstyle{remark}
\newtheorem{remark}[theorem]{Remark}
\newtheorem{example}[theorem]{Example}
\newtheorem{question}[theorem]{Question}

\numberwithin{equation}{section}

\def\id{\operatorname{id}}
\def\dom{\operatorname{dom}}
\def\im{\operatorname{Im}}
\def\Frac{\operatorname{Frac}}
\def\Const{\operatorname{Const}}
\def\spec{\operatorname{Spec}}
\def\span{\operatorname{span}}
\def\exc{\operatorname{Exc}}
\def\alg{\operatorname{alg}}
\def\sep{\operatorname{sep}}
\def\gal{\operatorname{Gal}}
\def\red{\operatorname{red}}
\def\mat{\operatorname{Mat}}
\def\divisor{\operatorname{div}}
\def\Fr{\operatorname{Fr}}

\title{Invariant hypersurfaces}

\author{Jason Bell}
\address{Jason Bell\\
University of Waterloo\\
Department of Pure Mathematics\\
200 University Avenue West\\
Waterloo, Ontario \  N2L 3G1\\
Canada}
\email{jpbell@uwaterloo.ca}

\author{Rahim Moosa}
\address{Rahim Moosa\\
University of Waterloo\\
Department of Pure Mathematics\\
200 University Avenue West\\
Waterloo, Ontario \  N2L 3G1\\
Canada}
\email{rmoosa@uwaterloo.ca}

\author{Adam Topaz}
\address{Adam Topaz\\
Mathematical and Statistical Sciences\\
University of Alberta\\
632 Central Academic Building\\
Edmonton, Alberta \ T6G 2G1\\
Canada }
\email{topaz@ualberta.ca}

%\thanks{{\em Acknowledgements}: }

\date{\today}

\begin{abstract}
The following theorem, which includes as very special cases results of Jouanolou and Hrushovski on algebraic $D$-varieties on the one hand, and of Cantat on rational dynamics on the other, is established:
Working over a field of characteristic zero, suppose $\phi_1,\phi_2:Z\to X$ are dominant rational maps from a (possibly nonreduced) irreducible scheme $Z$ of finite-type to an algebraic variety $X$, with the property that there are infinitely many hypersurfaces on~$X$ whose scheme-theoretic inverse images under $\phi_1$ and $\phi_2$ agree.
Then there is a nonconstant rational function $g$ on $X$ such that $g\phi_1=g\phi_2$.
In the case when $Z$ is also reduced the scheme-theoretic inverse image can be replaced by the proper transform.
A partial result is obtained in positive characteristic.
Applications include an extension of the Jouanolou-Hrushovski theorem to generalised algebraic $\mathcal D$-varieties and of Cantat's theorem to self-correspondences.
\end{abstract}


\thanks{{\em Acknowledgements}: The authors are grateful for the referee's careful reading of the original submission, leading to several improvements and corrections.
J. Bell and R. Moosa were partially supported by their NSERC Discovery Grants. A. Topaz was partially supported by EPSRC program grant EP/M024830/1,
the University of Alberta, and an NSERC discovery grant.}

\maketitle

\setcounter{tocdepth}{1}

\bigskip
\tableofcontents

\vfill\pagebreak

%\bigskip
\section{Introduction}

\noindent
Fix an algebraically closed field $K$ of characteristic zero.
The following is the main result of this paper:

\begin{theorem}
\label{main}
Suppose $X$ is an algebraic variety, $Z$ is an irreducible algebraic scheme, and $\phi_1,\phi_2:Z\to X$ are rational maps whose restrictions to $Z_{\red}$ are dominant, all over $K$.
Then the following are equivalent:
\begin{itemize}
\item[(1)]
There exist nonempty Zariski open subsets $V\subseteq Z$ and $U\subseteq X$ such that the restrictions $\phi_1^V,\phi^V_2:V\to U$ are dominant regular morphisms, and there exist infinitely many hypersurfaces $H$ on $U$ satisfying
$$(\phi_1^V)^{-1}(H)=(\phi^V_2)^{-1}(H).$$
\item[(2)]
There exists $g\in K(X)\setminus K$ such that $g\phi_1=g\phi_2$.
\end{itemize}
If $Z$ is reduced then these are also equivalent to:
\begin{itemize}
\item[(3)]
There exist infinitely many hypersurfaces $H$ on $X$ satisfying
$\phi_1^{*}H=\phi_2^{*}H$.
\end{itemize}
\end{theorem}
\noindent
Here, and throughout this paper, we only consider \emph{algebraic schemes}, i.e. separated schemes of finite type.
By an {\em algebraic variety} we mean an integral (so reduced and irreducible) algebraic scheme, and by a {\em hypersurface} we mean a Zariski closed subset of pure codimension one. If $\phi:Z\to X$ is a morphism of schemes and $H\subseteq X$ is a Zariski closed subset, then we use $\phi^{-1}(H)$ to denote the {\em scheme-theoretic inverse image} of $H$.
If $\phi:Z\to X$ is a dominant rational map of algebraic varieties then $\phi^*H$ denotes the {\em proper transform} of $H$, i.e. the union of those irreducible components of the Zariski closure of the set-theoretic inverse image of $H$ that project dominantly onto irreducible components of $H$.

As motivation, let us consider two well-known special cases of the theorem.

The first is from differential-algebraic geometry.
By an {\em algebraic $D$-variety} we mean an affine algebraic variety $X$ over $K$ equipped with a regular section to the tangent bundle, $s:X\to TX$.
A closed subvariety $H\subseteq X$ is a {\em $D$-subvariety} if $s\upharpoonright_H:H\to TH$.
Note that $s$ corresponds to a $K$-linear derivation $\delta$ on the co-ordinate ring $K[X]$, and that a $D$-subvariety corresponds to a $\delta$-ideal of $K[X]$.
Note also that the derivation $\delta$ extends uniquely to a derivation on the fraction field $K(X)$.
The following is a consequence of unpublished work of Hrushovski~\cite{hrushovski-jouanolou} in the mid-nineties on model-theoretic implications of a theorem of Jouanalou~\cite{jouanolou} on foliations from the seventies; see~\cite[Theorem 4.2]{freitag-moosa} for a published account.

\begin{corollary}[Jouanolou-Hrushovski]
\label{udij}
Suppose $(X,s)$ is an algebraic $D$-variety with infinitely many $D$-subvarieties of pure codimension one.
Then there exists $g\in K(X)\setminus K$ such that $\delta(g)=0$.
\end{corollary}

\begin{proof}
If  $X=\spec(R)$, apply Theorem~\ref{main} to
\begin{itemize}
\item
$Z=\spec(R[\epsilon]/\epsilon^2)$,
\item
$\phi_1:Z\to X$ the morphism induced by the $K$-algebra homomorphism from $R$ to $R[\epsilon]/\epsilon^2$ given by $r\mapsto r+\delta(r)\epsilon$, and 
\item
$\phi_2:Z\to X$ the morphism induced by the natural inclusion of $R$ in $R[\epsilon]/\epsilon^2$.
\end{itemize}
See~$\S$\ref{section-diff} for details.
\end{proof}

In fact, we obtain the Jouanolou-Hrushovski result for the general setting of ``algebraic $\mathcal D$-varieties", where the derivation is replaced by any system $\mathcal D$ of {\em generalised operators} in the sense of Moosa-Scanlon~\cite{paperC}.
This is Theorem~\ref{main-d} below.

We also recover from Theorem~\ref{main} a result in rational dynamics.
By a {\em rational dynamical system} we mean an algebraic variety $X$ over $K$ equipped with a dominant rational self-map $\phi:X\to X$.
A Zariski closed subset $H\subseteq X$ is {\em totally invariant} if $\phi^*H=H$.
The following is the algebraic case of Theorem~B in~\cite{cantat}. 

\begin{corollary}[Cantat~\cite{cantat}]
\label{cor-cantat}
Suppose $(X,\phi)$ is a rational dynamical system with infinitely many totally invariant hypersurfaces.
Then there exists $g\in K(X)\setminus K$ such that $g\phi=g$.
\end{corollary}

\begin{proof}
Apply Theorem~\ref{main} to
\begin{itemize}
\item
$Z=X$,
\item
$\phi_1=\phi$, and 
\item
$\phi_2=\id_X$.
\end{itemize}
See~$\S$\ref{section-red} for details.
\end{proof}

Again, we actually get more: we can replace the dominant rational self-map $\phi$ in the above corollary with an arbitrary self-correspondence.
This is Corollary~\ref{main-corr} below.
In fact, it is, we think, useful to view the data of Theorem~\ref{main}, namely the diagram
$$\xymatrix{
Z\ar[d]_{\phi_2} \ar[r]^{\phi_1}&X\\
X
}$$
as a generalised notion of self-correspondence on $X$, a self-correspondence that need not be reduced and need not be finite-to-finite.

Our theorem thus unifies these two well-known results, yielding at the same time natural generalisations in both cases.

A word about the proof of Theorem~\ref{main}.
Our approach is algebraic, thus differing significantly from the methods of Jouanalou-Hrushovski, and Cantat in the special cases.
We first reduce to a situation where everything is defined over a finitely generated subfield and the hypersurfaces have principal vanishing ideals.
In that setting our result appears as Theorem~\ref{main-algebra} below, whose proof is where the main technical work of the paper is done.
When $Z$ is reduced we follow to some extent the approach of~\cite[Theorem~1.2]{BRS} which is related to Cantat's theorem but obtained independently.
A separate argument (appearing in Section~\ref{section-red}) is required to replace the scheme-theoretic inverse image with the proper transform in the reduced case.
When $Z$ is non-reduced we concoct a derivation and rely on a refinement of~\cite[Proposition~6.10]{bllsm} that  is itself a refinement (but independent) of Jouanolou-Hrushovski.
Besides this use of~\cite{bllsm}, which is substantial, our proof of Theorem~\ref{main} is largely self-contained.

Once Theorem~\ref{main} is proved, we look closer in Section~\ref{section-terminal} at the case when $Z$ is reduced, and are lead to study the birational geometry of algebraic varieties equipped with a set of hypersurfaces.
More precisely, we consider the category whose objects are normal varieties $X$ equipped with a set of prime divisors $S$, and where a morphism $(X,S)\to(Y,T)$ is a dominant rational map $X\to Y$ with generic fibre irreducible, and such that, up to a finite set, $S$ is obtained as the proper transform of elements of $T$.
For this category, in the case when we are working over a field of finite transcendence degree, we give a geometric proof that every object $(X,S)$ admits a terminal morphism; one that factors through every other morphism originating at $(X,S)$.
See Theorem~\ref{terminal} below for a precise statement.
Combining this theorem -- which may be of independent interest -- with Cantat's theorem, we obtain a more conceptual alternative proof of Theorem~\ref{main} in the special case when $Z$ is reduced and $\phi_1,\phi_2$ have irreducible generic fibres.

In a final section we discuss what goes through in positive characteristic.
Because of our reliance on the characteristic zero differential-algebraic geometry of~\cite{bllsm} when $Z$ is non-reduced, we restrict our attention to the reduced case.
But even that -- namely, the equivalence of conditions~(2) and~(3) of Theorem~\ref{main} when $Z$ is reduced -- does not hold as stated in positive characteristic.
We do expect it to hold if we ask the generic fibres of $\phi_1$ and $\phi_2$ to be {\em geometrically reduced} -- something that is automatically satisfied in characteristic zero.
But we are only able to prove the equivalence if we add to~(3) the additional constraint that infinitely many of the invariant hypersurfaces $H$ are defined over (the separable closure of) a fixed finitely generated subfield.
This is Theorem~\ref{main-charp} below.

Our current methods have a couple of drawbacks that we present here as suggesting possibilities for future work.
The first is regarding effective uniform bounds.
Tracing through the proofs it is possible to compute explicitly a bound $N$ such that the ``infinitely many" in~(1) and~(3) of Theorem~\ref{main} can be replaced by ``more than~$N$".
But $N$ will depend not only on natural geometric invariants associated to the data, but also on the rank of the divisor class group of $X$ (over a minimal field of definition).
So these bounds are worse, less uniform, than those that arise in the special cases dealt with by the work of Jouanolou-Hrushovski and Cantat.
It would be useful to find an effective bound $N$ that remains constant as $Z,X,\phi_1,\phi_2$ vary in an algebraic family.

A second defficiency is that we are not able to work in the complex analytic setting.
The methods of Jouanolou-Hrushovski and Cantat, in contrast, extend to compact complex manifolds and meromorphic maps.
In particular, Cantat's results in~\cite{cantat} include as a special case Krasnov's theorem~\cite{krasnov} that a compact complex manifold without nonconstant meromorphic functions has only finitely many hypersurfaces.
A generalisation of Theorem~\ref{main} that includes complex analytic spaces would therefore be of significant interest.

\medskip
\noindent
{\em Throughout this paper all rings are assumed to be commutative, unitary, and all fields are of characteristic zero except in the final Section~\ref{section-charp}.}

\bigskip
\section{Some differential algebra preliminaries}

\noindent
By a {\em derivation} we mean a linear map $\delta:R\to S$, where $R\subseteq S$ is an extension of integral domains of characteristic zero, such that $\delta(ab)=a\delta(b)+\delta(a)b$ for all $a,b\in R$.
If $A\subseteq R$ is a subring, then we say that $\delta$ is {\em $A$-linear} to mean that $\delta(a)=0$ for all $a\in A$ -- which we note is equivalent to $\delta$ being a morphism of $A$-modules.
By the {\em constants} of the derivation $\delta:R\to S$ we mean the subring $R^\delta:=\{a\in R:\delta a=0\}$.
If $R=S$ then we call $(R,\delta)$ a {\em differential ring}.
A differential ring whose underlying ring is a field is called a {\em differential field}.

Here is a basic fact about derivations that we record now for later use, and that is deduced by a straightforward computation using the Leibniz rule.

\begin{fact}
\label{differentiatepoly}
Suppose $\delta:R\to S$ is a derivation, $P$ is a polynomial in $R[x_1,\dots,x_n]$, and $a=(a_1,\dots,a_n)\in R^n$.
Then 
\begin{equation*}
\delta P(a)=\sum_{j=1}^n\frac{\delta P}{\delta x_j}(a)\delta(a_j) +P^\delta(a)
\end{equation*}
where $P^\delta\in S[x_1,\dots,x_n]$ is obtained by applying $\delta$ to the coefficients of $P$.
\end{fact}

The following two lemmas are also very elementary and well-known.

\begin{lemma}
\label{deltafacts}
Suppose $K/k$ is a function field extension, $\delta$ is a $k$-linear derivation on~$K$, and $R\subseteq K$ is a finitely generated $k$-subalgebra.
Then there exists a finitely generated $k$-algebra extension $R'$ of $R$ in $K$ such that $\delta$ restricts to a differential ring structure on $R'$.
\end{lemma}

\begin{proof}
We may as well assume that $R=k[a]$ for some  $a=(a_1,\dots,a_n)$ such that $K=k(a)$.
Because of Fact~\ref{differentiatepoly}, it suffices to show that after possibly extending $a$ to a longer finite tuple $a'=(a_1,\dots,a_m)$ from $K$, and setting $R':=k[a']$, that $\delta a_i\in R'$ for all $i\leq m$.
If we write $\delta(a_i)=\frac{P_i(a)}{Q_i(a)}$ for $i\leq n$, then it is not hard to see, using Fact~\ref{differentiatepoly} again, as well as the quotient rule for derivations, that $a':=(a_1,\dots,a_n,\frac{1}{Q_1(a)},\dots,\frac{1}{Q_n(a)})$ works.
\end{proof}

\begin{lemma}
\label{constantalgebraic}
If $(L,\delta)\supseteq (K,\delta)$ is a differential field extension then $$(K^\delta)^{\alg}\cap L=K^{\alg}\cap L^\delta.$$
In particular, taking $L=K$, we have that $K^\delta$ is relatively algebraically closed in~$K$.
\end{lemma}

\begin{proof}
If $a\in L$ is algebraic over $K^\delta$ then we apply Fact~\ref{differentiatepoly} with $P(x)$ a minimal polynomial of $a$ over $K^\delta$ and conclude that $\frac{d P}{d x}(a)\delta(a)=0$.
But as $\deg P\geq 1$, we must have that $\frac{d P}{d x}\neq 0$ and of degree strictly less than $\deg P$, so that $\delta(a)=0$.

Conversely, if $a\in L^\delta$ is algebraic over $K$ then we apply Fact~\ref{differentiatepoly} with $P(x)$ the minimal monic polynomial of $a$ over $K$ and conclude that $P^\delta(a)=0$.
But as $P$ is monic $P^\delta$ will be of strictly smaller degree unless it is identically zero.
So it must be identically zero, implying that all of the coefficients of $P$ are in $K^\delta$, and hence that $a$ is algebraic over $K^\delta$.
\end{proof}

The following is maybe less widely known, but is a consequence of an argument appearing in~\cite{hrushovski-jouanolou}.
We give a proof here for the sake of completeness.

\begin{lemma}
\label{extend}
Suppose $F/k$ is a function field extension, $K/F$ is a field extension, and $\delta:F\to K$ is a $k$-linear derivation.
Then $\delta$ extends to a differential field structure on $K$ such that $K^\delta$ is algebraic over $F^\delta$.
\end{lemma}

\begin{proof}
Suppose $t\in K$ is transcendental over $F$.
For each $\gamma\in k$, consider the derivation $\delta_\gamma:F(t)\to K$ induced by $\delta_\gamma(t)=\gamma t$.
We claim that for some $\gamma\in k$, $F(t)^{\delta_\gamma}=F^\delta$.
Fix a sufficiently saturated differentially closed field $(\mathcal U,D)$ with constant field $\mathcal C$.
Each extension $\delta_\gamma:F(t)\to K$ embeds into $(\mathcal U,D)$.
Moreover, as these extensions all agree on $F$, we may as well assume that $F\subseteq\mathcal U$ and that these embeddings are over $F$.
So we get elements $t_\gamma\in\mathcal U$ and subfields $K_\gamma\subseteq\mathcal U$, such that $(F(t),K,\delta_\gamma)$ is isomorphic to $(F(t_\gamma),K_\gamma,D\upharpoonright_{F(t_\gamma)})$.
In particular, $D(t_\gamma)=\gamma t_\gamma$.
Now, if $F(t)^{\delta_\gamma}\neq F^\delta$, and we let $g\in F(t)^{\delta_\gamma}\setminus F$, then the image of $g$ in $(\mathcal U,D)$ is an element $h\in F(t_\gamma)^D\setminus F$.
It follows by Steinitz exchange that $t_\gamma\in F(h)^{\alg}$.
Writing the function field $F$ as $k(a)$, we have that $F(h)^{\alg}\subseteq\mathcal C(a)^{\alg}$.
That is, $t_\gamma$ is a solution to the equation $D x=\gamma x$ in $\mathcal C(a)^{\alg}$.
But the set of $\gamma\in \mathcal C$ such that $D x=\gamma x$ has a solution in fixed finite transcendence degree extension of $\mathcal C$ -- such as $\mathcal C(a)^{\alg}$ is -- forms a finite rank additive subgroup of $\mathcal C$.
This is an old result of Kolchin~\cite{kolchin68}, but see also~\cite[Fact 4.3]{freitag-moosa}.
Hence there must exist $\gamma\in k$ for which $D x=\gamma x$ has no solution in $\mathcal C(a)^{\alg}$.
For such a $\gamma$, $F(t)^{\delta_\gamma}=F^\delta$, as desired.

So, iterating this process, if we let $E$ be a transcendence basis for $K$ over $F$, then we can find an extension of $\delta$, which we will also call $\delta$, to $F(E)$ with no new constants.
Since $K$ is algebraic over $F(E)$ there is by Fact~\ref{differentiatepoly} a unique further extension of $\delta$ to $K$.
By Lemma~\ref{constantalgebraic}, $K^\delta$ is algebraic over $F(E)^\delta=F^\delta$.
\end{proof}

As discussed in the introduction, a special case of our main theorem is the finite-dimensional case of the Jouanalou-Hrushovski theorem on $D$-subvarieties of codimension one.
An algebraic proof of this finiteness theorem in the context of several derivations was given in~\cite{bllsm}.
We will rely on the following refinement of that result in the case of a single derivation.

\begin{proposition}
\label{jou}
Suppose $k$ is a finitely generated field, $A$ is a finitely generated $k$-algebra that is an integral domain, and $\delta:A\to A$ is a $k$-linear derivation.
Suppose that there exists an infinite sequence $(r_j:j<\omega)$ in $\Frac(A)$ such that $\delta(r_j)/r_j\in A$ for all $j<\omega$, and such that $(r_j:j<\omega)$ is multiplicatively independent modulo $(k^{\alg})^\times$.
Then there exists $g\in\Frac(A)^\delta\setminus k^{\alg}$.
In fact, if $G$ is the multiplicative group generated by $(r_j:j<\omega)$, then $G\cap\Frac(A)^\delta$ is nontrivial.
\end{proposition}

\begin{remark}
By the sequence being {\em multiplicatively independent modulo $(k^{\alg})^\times$} we mean that its image in $\Frac(A)^\times/(\Frac(A)^\times\cap k^{\alg})$ is multiplicatively independent.
In other words, no nontrivial product of integer powers of the $r_j$'s is in $k^{\alg}$.
\end{remark}

\begin{proof}
This is quite close to~\cite[Proposition~6.10]{bllsm}, but among the differences are that we are working over a finitely generated field rather than an uncountable algebraically closed field, and that the $r_j$ are coming from $\Frac(A)$ rather than from $A$ itself.
We have therefore something to do.

First, let us observe that we get the ``in fact" clause for free.
Indeed, letting $F:=\Frac(A)^\delta$, consider the finitely generated $F$-algebra $A'=FA$.
Then $F$ is again a finitely generated field, $\delta$ is an $F$-linear differential structure on $A'$, and $(r_j:j<\omega)$ in $\Frac(A')=\Frac(A)$ satisfies $\delta(r_j)/r_j\in A'$ for all $j<\omega$.
We can apply the main statement of the theorem -- which we are assuming we have proved -- to this context over $F$.
Since $\Frac(A')^\delta=F$, we must have that  $(r_j:j<\omega)$ is not multiplicatively independent modulo $F^{\alg}$.
But note that $F$ is relatively algebraically closed in $\Frac(A)$ by Lemma~\ref{constantalgebraic}.
So $(r_j:j<\omega)$ is not multiplicatively independent modulo $F$.
That is, some nontrivial product of integer powers of the $r_j$ is in $F$.
We have shown that $G\cap F$ is nontrivial, as desired.

So it suffices to prove that $\Frac(A)^\delta\not\subseteq k^{\alg}$.

Next, observe that we can always replace $A$ by a finitely generated localisation.
Indeed, this does not change the fraction field, and, since $\delta(\frac{1}{f})=-\frac{\delta f}{f^2}$, any such localisation is a differential subring of $\big(\Frac(A),\delta\big)$.
So we may assume that $A$ is integrally closed.
Moreover, as $k$ is a finitely generated field, some finitely generated localisation of $A$ is a unique factorisation domain -- this is by~\cite[Lemma~6.11]{bllsm} though one expects it to have appeared elsewhere and earlier.
So we may also assume that $A$ is a UFD.

Consider $k':=\Frac(A)\cap k^{\alg}$.
By integral closedness, $k'\subseteq A$.
Moroever, $\delta$ is $k'$-linear.
The hypotheses of the theorem hold for $(A,k')$, and  if the conclusion were true for $(A,k')$ then it would be true of $(A,k)$.
That is, it suffices to prove the theorem for $k'$ in place of $k$.
So we may also assume $\Frac(A)\cap k^{\alg}=k$.

Next, we move the $r_j$ into $A$ itself.
For each $j<\omega$, write $r_j=\frac{c_j}{d_j}$ where $c_j,d_j\in A$ are coprime.
Since $\frac{\delta(r_j)}{r_j}=\frac{\delta(c_j)}{c_j}-\frac{\delta(d_j)}{d_j}$, it follows that $\frac{\delta(c_j)d_j}{c_j}\in A$.
Coprimality of $c_j$ and $d_j$ in $A$ then implies $\frac{\delta(c_j)}{c_j}\in A$.
A symmetric argument shows that $\frac{\delta(d_j)}{d_j}\in A$.
Note that the multiplicative group generated by $\{c_j,d_j:j<\omega\}$ contains that generated by $\{r_j:j<\omega\}$, so the former must also have infinite rank modulo its intersection with $k^{\alg}$.
We can therefore find in $A$ a sequence $(a_j:j<\omega)$ that is multiplicatively independent modulo $(k^{\alg})^\times$ and such that $\frac{\delta(a_j)}{a_j}\in A$ for all $j<\omega$.

Let $K$ be an uncountable algebraically closed field extending $k$.
It follows that $A_K:=A\otimes_kK$ is an integrally closed domain extending~$A$, finitely generated over $K$, and with the property that $\Frac(A)\cap K=k$ in $\Frac(A_K)$.
Hence, no nontrivial product of integer powers of the $a_j$ is in~$K$ either.
Moreover $\delta$ extends to a $K$-linear derivation on $\Frac(A_K)$ and $\frac{\delta(a_j)}{a_j}\in A_K$ for all $j<\omega$.
Proposition~6.10 of~\cite{bllsm} now applies and we obtain an element $g\in \Frac(A_K)^\delta\setminus K$.

At this point we can use a specialisation argument, or, as we prefer to do, the model-completeness of the first order theory of algebraically closed fields, to see that $\Frac(A_{k^{\alg}})^\delta\setminus k^{\alg}\neq\emptyset$ where $A_{k^{\alg}}:=A\otimes_kk^{\alg}$.
Indeed, letting $(X,s)$ be the $D$-variety associated to $(A,\delta)$, we denote by $(X_K,s)$ the base extension to~$K$, that is, the $D$-variety associated to $(A_K,\delta)$.
Hence, $g:(X_K,s)\to(\mathbb A_K^1,0)$ is a nonconstant rational map over $K$.
This is a first order expressible property over $k$ of the parameters in $K$ over which $g$ is defined.
As $k^{\alg}$ is an elementary substructure of $K$, we thus obtain nonconstant $g':(X_{k^{\alg}},s)\to(\mathbb A_{k^{\alg}}^1,0)$ over $k^{\alg}$.
That is, $g'\in\Frac(A_{k^{\alg}})^\delta\setminus k^{\alg}$.
Now, $\Frac(A_{k^{\alg}})^\delta$ is algebraic over $\Frac(A)^{\delta}$, so the canonical parameter for the finite set of Galois conjugates of $g'$ is a tuple from $\Frac(A)^\delta$ not all of whose co-ordinates can be in~$k^{\alg}$ since $g$ isn't.
Hence $\Frac(A)^\delta\setminus k^{\alg}\neq\emptyset$.
\end{proof}

\bigskip
\section{The principal algebraic statement}
\label{section-algebra}

\noindent
The key step in our proof of Theorem~\ref{main} will be the following statement in commutative algebra.
It is given here in slightly greater generality than necessary; we will only apply it in the case that the nilradical of $S$ is prime, and the reader is invited to make this assumption, and thereby remove a few of the technicalities, if he or she desires.

\begin{theorem}
\label{main-algebra}
Suppose
\begin{itemize}
\item
$k$ is a finitely generated field of characteristic zero,
\item
$R$ is a finitely generated $k$-algebra that is an integral domain and such that $k$ is relatively algebraically closed in $\Frac(R)$,
\item
$S$ is a finitely generated $k$-algebra, such that $k$ is relatively algebraically closed in $\Frac(S/P)$ for every minimal prime ideal $P$ of $S$, and
\item
$f_1,f_2:R\to S$ are $k$-algebra homomorphisms that take nonzero elements of $R$ to regular elements -- that is, non zero-divisors --  of $S$.
\end{itemize}
Suppose there exists a sequence of nonzero elements $(a_j:j<\omega)$ in $R$ that is multiplicatively independent modulo $k^\times$, and such that $f_1(a_j)S=f_2(a_j)S$ for all $j<\omega$.
Then there exists $g\in\Frac(R)\setminus k$ such that $f_1(g)=f_2(g)$.

In fact, if we let $F$ be the subfield of $\Frac(R)$ on which $f_1$ and $f_2$ agree, and we let $G$ be the subgroup of $\Frac(R)^\times$ generated by $(a_j:j<\omega)$, then $G\cap F$ is nontrivial.
\end{theorem}

\begin{remark}
\begin{itemize}
\item[(a)]
The assumptions on $f_1$ and $f_2$ imply that they extend uniquely to embeddings of $\Frac(R)$ into $\Frac(S)$, where by $\Frac(S)$ we mean the localisation of $S$ at the set of all regular elements.
It is with respect to these extensions that we mean $f_1(g)=f_2(g)$ in the conclusion of the theorem.
\item[(b)]
Any nontrivial element of $G$ is necessarily transcendental over $k$; this follows from the multiplicative independence of $(a_j:j<\omega)$ modulo $k^\times$ together with the fact that $k$ is relatively algebraically closed in $\Frac(R)$.
\end{itemize}
\end{remark}

\begin{proof}
Let us first consider the case when $S$ is a reduced ring.

Since $ f_1(a_j)S= f_2(a_j)S$, for each $j<\omega$ 
there is a unit $u_j$ in $S$ such that $f_1(a_j)=u_j f_2(a_j)$.
Let $P_1,\dots,P_\ell$ be the minimal primes of $S$, and denote by $S_\mu:=S/P_\mu$ the corresponding integral domain for each $\mu=1,\dots,\ell$.
Let
$$\bar u_j:=(u_j+P_1,\dots,u_j+P_\ell)\in S_1^\times\times\cdots\times S_\ell^\times.$$
Now, as $k$ is relatively algebraically closed in $S_\mu$, each $S_\mu^\times/k^\times$ is a finitely generated group -- see, for example, \cite[Corollary~2.7.3]{Lang}.
Hence, $S_1^\times/k^\times\times\cdots\times S_\ell^\times/k^\times$ is finitely generated.
It follows that for some $N>0$ and all $r>0$, 
$$\bar u_{(r-1)N+1}^{k_{r,1}}\cdot \bar u_{(r-1)N+2}^{k_{r,2}}\cdots \bar u_{rN}^{k_{r,N}}=\lambda_r$$
for some $\lambda_r=(\lambda_{r,1},\dots,\lambda_{r,\ell})\in (k^\times)^\ell$ and some $k_{r,1},\dots,k_{r,N}\in\mathbb Z$ not all zero.
(Note that if you assumed in the theorem that the nilradical of $S$ was prime then here $S$ is already a domain and $\ell=1$ with $P_1=(0)$.)

Let $f_{1,\mu},f_{2,\mu}:R\to S_\mu$ be the $k$-algebra homomorphisms induced by $f_1$ and $f_2$ for $\mu=1,\dots,\ell$.
By assumption, they are still injective.
Consider, for each $r>0$,
$$h_r:=a_{(r-1)N+1}^{k_{r,1}}\cdot a_{(r-1)N+2}^{k_{r,2}}\cdots a_{rN}^{k_{r,N}}$$
in $\Frac(R)$.
By construction $f_{1,\mu}(h_r)=\lambda_{r,\mu} f_{2,\mu}(h_r)$.

Letting $m$ be greater than the transcendence degree of $\Frac(R)$ over $k$, we get that $\{h_1,\dots,h_m\}$ is algebraically dependent over $k$.
Let $\sum c_{i_1,\dots,i_m}h_1^{i_1}\cdots h_m^{i_m}=0$ be a nontrivial algebraic relation over $k$ with a minimal number of nonzero coefficients.
Note that as none of the $h_r$ are zero, there are at least two nonzero coefficients in this relation.
Fixing  $\mu=1,\dots,\ell$ and applying $f_{1,\mu}$ to this we get
$$\sum c_{i_1,\dots,i_m}\lambda_{1,\mu}^{i_1}\cdots\lambda_{m,\mu}^{i_m}f_{2,\mu}(h_1)^{i_1}\cdots f_{2,\mu}(h_m)^{i_m}=0$$
while applying $f_{2,\mu}$ yields
$$\sum c_{i_1,\dots,i_m}f_{2,\mu}(h_1)^{i_1}\cdots f_{2,\mu}(h_m)^{i_m}=0.$$
Suppose that $\lambda_{1,\mu}^{i_1}\cdots\lambda_{m,\mu}^{i_m}\neq\lambda_{1,\mu}^{j_1}\cdots\lambda_{m,\mu}^{j_m}$ for some distinct (nonzero) $c_{i_1,\dots,i_m}$ and $c_{j_1,\dots,j_m}$.
Manipulating these two equations and then taking $f_{2,\mu}^{-1}$ we would get a relation among the $h_1,\dots, h_m$ with fewer nonzero coefficients.
As this is impossible by minimality, it must be that $\lambda_{1,\mu}^{i_1}\cdots\lambda_{m,\mu}^{i_m}=\lambda_{1,\mu}^{j_1}\cdots\lambda_{m,\mu}^{j_m}$ whenever $c_{i_1,\dots,i_m}$ and $c_{j_1,\dots,j_m}$ are nonzero.
But then, fixing $(i_1,\dots,i_m)\neq(j_1,\dots,j_m)$ with $c_{i_1,\dots,i_m}$ and $c_{j_1,\dots,j_m}$ nonzero, we get that
$$f_{1,\mu}(h_1^{i_1-j_1}\cdots h_m^{i_m-j_m})=f_{2,\mu}(h_1^{i_1-j_1}\cdots h_m^{i_m-j_m}).$$
Setting $g:=h_1^{i_1-j_1}\cdots h_m^{i_m-j_m}$ we have that $f_{1,\mu}(g)=f_{2,\mu}(g)$ for all $\mu=1,\dots,\ell$.
Since $S$ is reduced, $P_1\cap\dots\cap P_\ell=(0)$, and hence $f_1(g)=f_2(g)$.
That is, $g\in G\cap F$.
It remains only to verify that $g\neq 1$.
Since $i_r-j_r\neq 0$ for some $r=1,\dots,m$, and the corresponding $k_{r,1},\dots,k_{r,N}$ are not all zero, $g$ is a nontrivial product of integer powers of the $a_i$.
By multiplicative independence modulo $k^\times$, $g\neq 1$.

Now we deal with the case when the nilradical $N$ of $S$ is nontrivial.
For any ideal $I\leq N$, let $f^I_\nu:R\to S/I$ be the composition of $f_\nu$ with $S\to S/I$.
Notice that since $I\leq N$ the image of a regular element in $S$ remains regular in $S/I$.
So $f^I_1, f^I_2$ extend to embeddings of $\Frac(R)$ in $\Frac(S/I)$.
Consider the subfield
$$F_I:=\{g\in \Frac(R): f^I_1(g)= f^I_2(g)\}.$$
Note that $F=F_{(0)}$.
Note also that $G/G\cap F_N$ is of finite rank.
Indeed, otherwise we would have a subsequence $(b_j:j<\omega)$ of $(a_j:j<\omega)$ that is multiplicatively independent modulo $F_N$.
But the reduced case applied to this subsequence, which is {\em a fortiori} multiplicatively independent modulo $k^\times$, implies that $\langle b_j:j<\omega\rangle\cap F_N$ is nontrivial, contradicting the multiplicative independence modulo $F_N$.

%We assume, toward a contradiction, that the conclusion of the theorem fails and no such $g$ exists. So we are assuming that $F_{(0)}\cap G=k^\times$. In particular, $G/(F_{(0)}\cap G)$ is of infinite rank.

\begin{reduction}
\label{maxass}
It suffices to prove the theorem under the assumption that for any nonzero ideal $I\leq N$, $G/(G\cap F_I)$ is of finite rank.
\end{reduction}

\begin{proof}[Proof of Reduction~\ref{maxass}]
Assume we have proven the theorem under this additional condition.
Set $\mathcal I:=\{I\leq N:G/(G\cap F_I)\text{ is of infinite rank}\}$.
Suppose, toward a contradiction, that $\mathcal I$ is nonempty.
By noetherianity there is a maximal element $J\in\mathcal I$.
Let $(b_j:j<\omega)$ be a subsequence of $(a_i:i<\omega)$ that is multiplicatively independent modulo $F_J$.
Notice that the assumptions of the theorem remain true of $f^J_1,f^J_2$ and $(b_\ell:\ell<\omega)$.
Moreover, the condition of the claim is true of $f^J_1,f^J_2$ and $(b_\ell:\ell<\omega)$ by maximal choice of $J$.
Applying the theorem to $f^J_1,f^J_2$ and $(b_\ell:\ell<\omega)$, we get that $\langle b_\ell:\ell<\omega\rangle\cap F_J$ is nontrivial.
But this contradicts the multiplicative independence of $(b_\ell:\ell<\omega)$ over $F_J$.

So $\mathcal I$ is empty.
In particular, $G/(G\cap F)$ is of finite rank.
But then certainly $G\cap F$ is nontrivial, as desired.
\end{proof}

\begin{reduction}
\label{xass}
It suffices to prove the theorem under the additional assumption that there exists a nonzero $x\in N$ such that $x^2=0$, $Q:=\operatorname{ann}(x)$ is prime, and if $(x)=Q_1\cap\cdots\cap Q_\ell$ is the primary decomposition of $(x)$ then $Q_1\cup\cdots\cup Q_\ell$ contains no regular element of $S$.
\end{reduction}

\begin{proof}[Proof of Reduction~\ref{xass}]
As $N$ is not trivial, let $x\in N$ be a nonzero element with maximal annihilator.
Then $x^2=0$ and primality of $Q$ follow.
For each of the finitely many ideals in the primary decomposition of $(x)$ that contain a regular element, choose one. 
Let $S'$ be the extension of $S$ obtained by inverting the product of these finitely many regular elements.
Then $S'$ is still finitely generated and we can now apply the theorem with $S'$ in place of $S$, noting that $x$ is still nonzero in $S'$, $x^2=0$ remains true, $\operatorname{ann}_{S'}(x)=QS'$ is still prime, and now the primary ideals appearing in the decomposition of $xS'$ all have no regular elements.
We therefore obtain nonconstant $g\in \Frac(R)$ such that $f_1(g)=f_2(g)$ in $\Frac(S')=\Frac(S)$.
\end{proof}

Let $x$ be as in Reduction~\ref{xass}.
By Reduction~\ref{maxass}, we have that $G/(G\cap F_{(x)})$ is of finite rank.
There must exist a sequence $(b_j:j<\omega)$ in $G\cap F_{(x)}$ that is multiplicatively independent modulo $k^\times$.
Being in $G$ implies that each $b_j$ is product of integer powers of some $a_i$'s.
Since $f_1(a_i)S=f_2(a_i)S$ we have that $f_1(a_i)$ is a multiple of $f_2(a_i)$ by a unit in $S^\times$.
The same is therefore true of $b_j$.
That is,
\begin{equation}
\label{ratiounit}
f_1(b_j)=u_jf_2(b_j)\ \text{ for some unit }u_j\in S^\times\ \ \text{ for all }j<\omega.
\end{equation}
On the other hand, being in $F_{(x)}$ means that
\begin{equation}
\label{inx}
f_1(b_j)-f_2(b_j)\in x\Frac(S)\ \ \text{ for all }j<\omega.
\end{equation}
Let $T:=k[(b_j:j<\omega)]$ be the $k$-subalgebra of $\Frac(R)$ generated by these elements, and consider the $k$-linear map $f_1-f_2$ restricted to $T$.
Using the fact that the $f_i$ are ring homomorphisms, and developing, one obtains the
the following twisted Leibniz rule:
\begin{equation}
\label{leibniz}
(f_1-f_2)(uv)=(f_1-f_2)(u)f_1(v)+f_2(u)(f_1-f_2)(v).
\end{equation}
Moreover,
\begin{equation}
\label{imageinx}
(f_1-f_2)(T)\subseteq x\Frac(S).
\end{equation}
Indeed, equation~(\ref{inx}) tells us that $(f_1-f_2)$ takes the generators of $T$ into $x\Frac(S)$, and this property is clearly linear.
So assuming that $(f_1-f_2)(u), (f_1-f_2)(v)\in x\Frac(S)$, it suffices to show that $(f_1-f_2)(uv)\in x\Frac(S)$.
This follows immediately from~(\ref{leibniz}) above.

Let $\pi_Q:\Frac(S)\to\Frac(S/Q)$ be induced by the quotient $S\to S/Q$.
Then $\pi_Q\circ f_1,\pi_Q\circ f_2:\Frac(R)\to\Frac(S/Q)$ are embeddings of fields.
Since $x^2=0$, we have $x\in Q$, and so equation~(\ref{imageinx}) tells us that $\pi_Q\circ f_1$ and $\pi_Q\circ f_2$ agree on $T$.
We use this embedding to view $T\subseteq\Frac(S/Q)$.

\begin{claim}
\label{mu}
There exists a $k$-linear derivation $\delta$ on $\Frac(S/Q)$ satisfying:
\begin{itemize}
\item[(i)]
For all $t\in T$ and $s\in\Frac(S)$, $\pi_Q(s)=\delta(t)$ if and only if $f_1(t)-f_2(t)=xs$.
\item[(ii)]
The constant field $\Frac(S/Q)^\delta$ is algebraic over $\Frac(T)^\delta$.
\item[(iii)]$\Frac(T)^\delta\subseteq F$.
\end{itemize}
\end{claim}

\begin{proof}[Proof of Claim~\ref{mu}]
We first find a derivation $\delta:T\to\Frac(S/Q)$ with property~(i).
Given $t\in T$, we know by~(\ref{imageinx}) that $f_1(t)-f_2(t)=xs$ for some $s\in \Frac(S)$.
Now, for any $s'\in\Frac(S)$ we have
\begin{eqnarray*}
f_1(t)-f_2(t)=xs'
&\iff&
(s-s')x=0\\
&\iff&
s-s'\in Q\Frac(S)\ \ \text{ since $Q=\operatorname{ann}(x)$}\\
&\iff&
\pi_Q(s)=\pi_Q(s').
\end{eqnarray*}
So we can define $\delta(t):=\pi_Q(s)$, and it will have the desired property.
That $\delta$ is $k$-linear is clear from the construction.
That it is a derivation follows from~(\ref{leibniz}).
Indeed, given $u,v\in T$, let $s_u,s_v\in\Frac(S)$ be such that $(f_1-f_2)(u)=xs_u$ and $(f_1-f_2)(v)=xs_v$.
Then~(\ref{leibniz}) along with the construction of $\delta$ gives us:
\begin{eqnarray*}
\delta(uv)
&=&
\pi_Q\big(s_uf_1(v)+f_2(u)s_v)\\
&=&
\pi_Q(s_u)v+u\pi_Q(s_v)\ \ \ \ \ \text{ by our identification of $T\subseteq\Frac(S/Q)$}\\
&=&
\delta(u)v+u\delta(v)
\end{eqnarray*}
as desired.

Now, there is a unique extension of $\delta$ to $\Frac(T)$ using the usual quotient rule: $\delta\left(\frac{u}{v}\right):=\frac{v\delta u-u\delta v}{v^2}$.
Note that  $\Frac(T)$ is finitely generated over $k$ as it is a subextension of the finitely generated extension $\Frac(R)$. 
So we can apply Lemma~\ref{extend} and extend $\delta$ further to a derivation $\delta:\Frac(S/Q)\to\Frac(S/Q)$ whose constant field is algebraic over the constants in $\Frac(T)$.
That is, it satisfies property~(ii).

Finally we show~(iii).
Suppose $g=\frac{u}{v}\in\Frac(T)$ and $\delta(g)=0$.
This means $v\delta(u)=u\delta(v)$.
Letting $s_u,s_v\in\Frac(S)$ be such that $(f_1-f_2)(u)=xs_u$ and $(f_1-f_2)(v)=xs_v$, we have by~(i) that $\pi_Q(s_u)=\delta(u)$ and $\pi_Q(s_v)=\delta(v)$.
So, $\pi_Q\big(f_1(v)s_u-f_1(u)s_v\big)=v\delta(u)-u\delta(v)=0$ so that $f_1(v)s_u-f_1(u)s_v\in Q\Frac(S)$.
Hence
$0=\big(f_1(v)s_u-f_1(u)s_v\big)x=f_1(v)(f_1-f_2)(u)-f_1(u)(f_1-f_2)(v)$.
That is,
$\frac{(f_1-f_2)(u)}{f_1(u)}=\frac{(f_1-f_2)(v)}{f_1(v)}$, which implies that $1-\frac{f_2(u)}{f_1(u)}=1-\frac{f_2(v)}{f_1(v)}$, and so $\frac{f_1(u)}{f_1(v)}=\frac{f_2(u)}{f_2(v)}$.
That is, $f_1(g)=f_2(g)$, as desired.
\end{proof}

Let $M:=\{s\in\Frac(S):sx\in S\}$.
This is an $S$-submodule of $\Frac(S)$ that contains $Q\Frac(S)$.

\begin{claim}
\label{fg}
$\pi_Q(M)=S/Q$.
\end{claim}
\begin{proof}[Proof of Claim~\ref{fg}]
As $M$ contains $S$, it suffices to show that $\pi_Q(M)\subseteq S/Q$.
Suppose $\frac{c}{d}\in M$.
So $cx=yd$ for some $y\in S$.
Let $(x)=Q_1\cap\cdots\cap Q_\ell$ be the primary decomposition of $(x)$ in $S$.
It follows that for each $i=1,\dots,\ell$, $yd\in Q_i$.
Since $d$ is regular, Reduction~\ref{xass} implies no power of $d$ can be in $Q_i$.
Hence $y\in Q_i$ for all $i$.
So $y=y'x$ for some $y'\in S$.
Hence $cx=dy'x$, so that $(c-dy')\in\operatorname{ann}(x)=Q$.
It follows that $\pi_Q\left(\frac{c}{d}\right)=\pi_Q(y')\in S/Q$.
\end{proof}

Let $A$ be a finitely generated $k$-subalgebra of $\Frac(S/Q)$ that contains $S/Q$ and is preserved by the derivation $\delta$ obtained in Claim~\ref{mu} -- by Fact~\ref{deltafacts} this is possible.
We show that for all $j<\omega$, $\frac{\delta(b_j)}{b_j}\in A$.
Choose $s_j\in\Frac(S)$ such that
$$f_1(b_j)-f_2(b_j)=xs_j.$$
By~(\ref{ratiounit}) we also have $f_1(b_j)-f_2(b_j)=f_2(b_j)(u_j-1)$ where $u_j\in S^\times$.
So $\frac{s_j}{f_2(b_j)}x\in S$, and hence $\frac{s_j}{f_2(b_j)}\in M$.
Applying $\pi_Q$ we get by Claim~\ref{mu}(i) that $\frac{\delta(b_j)}{b_j}\in \pi_Q(M)$.
Now by Claim~\ref{fg}, $\frac{\delta(b_j)}{b_j}\in S/Q\subseteq A$.

To recap then, $(b_j:j<\omega)$ is a sequence in $\Frac(R)$ that is multiplicatively independent modulo $k^\times$, and hence modulo $(k^{\alg})^\times$ since $k$ is relatively algebraically closed in $R$.
By the discussion preceding Claim~\ref{mu} we view $(b_j:j<\omega)$ as a sequence in $\Frac(A)=\Frac(S/Q)$, and as such have just shown that it satisfies $\frac{\delta(b_j)}{b_j}\in A$ for all $j<\omega$.
We are thus in the context of Proposition~\ref{jou}, and we can conclude that $\langle b_j:j<\omega\rangle \cap \Frac(A)^\delta$ is nontrivial.
But~\ref{mu}(ii) tells us that $\Frac(A)^\delta$ is algebraic over $\Frac(T)^\delta$, and~\ref{mu}(iii) says the latter is in $F$.
It follows that $\langle b_j:j<\omega\rangle\cap F^{\alg}$ is nontrivial.

Suppose $g\in \langle b_j:j<\omega\rangle\cap F^{\alg}$ is nontrivial.
We claim, finally, that $g\in F$.
Indeed, suppose toward a contradiction that for some $m>1$,
$$P(X)=X^m+{c_{m-1}}X^{m-1}+\cdots+c_0$$
is the minimal polynomial of $g$ over $F$.
As $f_1$ and $f_2$ agree on $F$ we may as well identify $F$ with its image in $\Frac(S)$ so that $f_1,f_2$ become $F$-linear.
Applying $f_\nu$ to $P(g)=0$ for $\nu=1,2$ yields
\begin{equation}
\label{nupoly}
f_\nu(g)^m+{c_{m-1}}f_\nu(g)^{m-1}+\cdots+c_0=0
\end{equation}
in $\Frac(S)$.
Since $g\in F_{(x)}$ we have that $f_1(g)=f_2(g)+sx$ for some $s\in\Frac(S)$.
Substituting this into~(\ref{nupoly}) for $\nu=1$ we get
\begin{eqnarray*}
0
&=&
(f_2(g)+sx)^m+{c_{m-1}}(f_2(g)+sx)^{m-1}+\cdots+c_0\\
&=&
\big(f_2(g)^m+{c_{m-1}}f_2(g)^{m-1}+\cdots+c_0\big) +\\
& &
sxmf_2(g)^{m-1}+sx(m-1)c_{m-1}f_2(g)^{m-2}+\cdots+sxc_1\\
&=&
sxmf_2(g)^{m-1}+sx(m-1)c_{m-1}f_2(g)^{m-2}+\cdots+sxc_1\\
&=&
sx \cdot f_2\big(mg^{m-1}+(m-1)c_{m-1}g^{m-2}+\cdots+c_1\big)
\end{eqnarray*}
where the second equality uses~(\ref{nupoly}) for $\nu=2$, and the fact that $x^2=0$.
So, if we let $g':=mg^{m-1}+(m-1)c_{m-1}g^{m-2}+\cdots+c_1$, then $sxf_2(g')=0$.
Note that $g'\neq 0$ by the minimality of the degree $m$, and hence $f_2(g')$ is regular in $S$.
It follows that $sx=0$.
But this means that $f_1(g)=f_2(g)$, so that $g\in F$, contradicting $m>1$.

We have proved that $\langle b_j:j<\omega\rangle\cap F$, and hence $G\cap F$, is nontrivial.
\end{proof}

\bigskip
\section{Proof of the main theorem}

\noindent
We now deduce the main part of Theorem~\ref{main} as stated in the introduction from the algebraic statement given in Theorem~\ref{main-algebra}.
We will deal with the rest of the statement, namely the improvement in the reduced case, in $\S$\ref{section-red} below.

\begin{theorem}
\label{main-nr}
Suppose $X$ is an algebraic variety, $Z$ is an irreducible algebraic scheme, and $\phi_1,\phi_2:Z\to X$ are rational maps whose restrictions to $Z_{\red}$ are dominant, all over an algebraically closed field $K$.
Then the following are equivalent:
\begin{itemize}
\item[(1)]
There exist nonempty Zariski open subsets $V\subseteq Z$ and $U\subseteq X$ such that the restrictions $\phi_1^V,\phi^V_2:V\to U$ are dominant regular morphisms, and there exist infinitely many hypersurfaces $H$ on $U$ satisfying
$$(\phi_1^V)^{-1}(H)=(\phi^V_2)^{-1}(H).$$
\item[(2)]
There exists $g\in K(X)\setminus K$ such that $g\phi_1=g\phi_2$.
\end{itemize}
\end{theorem}

\begin{proof}
That~(2) implies~(1) is more or less clear: we can choose nonempty Zariski open sets $V\subseteq Z$ and $U\subseteq X$ such that the restrictions $\phi^V_1,\phi^V_2:V\to U$ are dominant regular morphisms and such that $g:U\to \mathbb A^1$ is a nonconstant morphism to the affine line.
We have the commuting diagram
$$\xymatrix{
V\ar[r]^{\phi_1^V}\ar[d]_{\phi_2^V} &U\ar[d]^{g}\\
U\ar[r]_{g}&\mathbb A^1
}$$
so that level sets of $g$ over the $K$-points of $\mathbb A^1$ yield infinitely many hypersurfaces $H$ on $U$ satisfying $(\phi_1^V)^{-1}(H)=(\phi^V_2)^{-1}(H)$.

Assume that~(1) holds.

Let $k\subseteq K$ be a finitely generated subfield over which $Z, X, V, U, \phi_1, \phi_2$ are defined.
That is, $X=X_k\times_kK$ for some geometrically irreducible algebraic $k$-variety $X_k$, and $Z=Z_k\times_kK$ where $(Z_k)_{\red}$ is a geometrically irreducible algebraic $k$-variety.
We have similar descent statements to $k$ for $V, U, \phi_1,\phi_2$ as well.

We first claim that $k$ can be chosen so that there are infinitely many hypersurfaces $H$ on $U$ defined over $k$ satisfying $(\phi_1^V)^{-1}(H)=(\phi^V_2)^{-1}(H)$.
Indeed, fix $k$ and suppose there exists such a hypersurface $H$ that is not defined over $k^{\alg}$.
Then $H$ is defined over a finitely generated nonalgebraic extension $L$ of~$k$.
Now $\operatorname{Aut}(L^{\alg}/k)$ acts naturally on the whole situation, and there are infinitely many $\operatorname{Aut}(L^{\alg}/k)$-conjugates of $H$ in $U$.
All these conjugates are defined over $L^{\alg}$ and satisfy the property that their inverse images under $\phi^V_1$ and $\phi^V_2$ agree.
So choosing $L$ instead of~$k$, we may as well assume that we have infinitely many such hypersurfaces over~$k^{\alg}$ to start with.
Replacing each of these with the union of their conjugates under the action of $\gal(k)$, we may in fact assume they are over $k$ itself.

Suppose therefore that $(H_j:j<\omega)$ is an infinite sequence of hypersurfaces over $k$ on $U$ with $(\phi^V_1)^{-1}(H_j)=(\phi^V_2)^{-1}(H_j)$ for all $j<\omega$, and such that $H_j\not\subseteq\bigcup_{i<j}H_i$.

Replacing $V$ and $U$ by smaller nonempty Zariski open subsets, we may assume $U=\spec(R_K)$ and $V=\spec(S_K)$ where $R$ is a finitely generated $k$-algebra that is an integral domain, $S$ is a finitely generated $k$-algebra whose nilradical $N$ is prime, $R_K:=R\otimes_kK$ and $S_K:=S\otimes_kK$, and $\phi^V_1,\phi^V_2$ are induced by $k$-algebra homomorphisms $f_1,f_2:R\to S$.
The assumption that the $\phi_i$ restrict to dominant rational maps on $Z_{\red}$ tells us that the $f_i$ composed with the  quotient by $N$ are injective.
(Hence, $f_1,f_2$ themselves are embeddings.)
Note that as $X_k$ is geometrically irreducible, $k$ is relatively algebraically closed in $\Frac(R)$.
Similarly, as $(Z_k)_{\red}$ is geometrically irreducible, $k$ is relatively algebraically closed in $\Frac(S/N)$.

Now, as $k$ is a finitely generated field, the localisation of $R$ at some nonzero element is a unique factorisation domain -- this is by~\cite[Lemma~6.11]{bllsm}.
So we may assume that $R$ is already a UFD.
The vanishing ideals $I(H_j)$ are of the form $I_jR_K$ where $I_j$ is a (radical) height one ideal in $R$, and hence of the form $I_j=a_jR$ for some sequence $(a_j:j<\omega)$ in $R$.
The scheme-theoretic inverse images $(\phi_\nu^V)^{-1}(H_j)$ are by definition given by the ideals $f_\nu(I_j)S_K$, for $\nu=1,2$.
That $(\phi^V_1)^{-1}(H_j)=(\phi^V_2)^{-1}(H_j)$ therefore implies that $f_1(a_j)S=f_2(a_j)S$ for all $j<\omega$.
Moreover, since $ H_j\not\subseteq\bigcup_{i<j}H_i$, each $a_j$ has an irreducible factor that does not appear in $a_i$ for $i<j$, and so no nontrivial product of integer powers of the $a_j$ can be a constant in $\Frac(R)$.
That is, the hypotheses of Theorem~\ref{main-algebra} are satisfied, and there must exist $g\in\Frac(R)\setminus k$ such that $f_1(g)=f_2(g)$.
Note that $g\in\Frac(R_K)=K(X)$ and $g\phi_1=g\phi_2$.
In $K(X)=\Frac(R)\otimes_kK$ the intersection of $\Frac(R)$ and $K$ is $k$, so we have that $g\notin K$.
This proves~(2).
\end{proof}

\bigskip
\section{An application to algebraic $\mathcal D$-varieties}
\label{section-diff}

\noindent
In this section we specialise Theorem~\ref{main-nr} to the differential context to see how we recover the finite-dimensional Jouanalou-Hrushovski theorem.
In fact we work rather more generally in a setting that appears in the work of the second author and Thomas Scanlon~\cite{paperC} toward the model theory of fields equipped with a general class of operators.
We will thus obtain a Jouanalou-Hrushovski type theorem for these generalised operators.

The setting is as follows.
Fix an algebraically closed field $K$ of characteristic zero.
Let $\mathcal D$ denote the following fixed data:
\begin{itemize}
\item
a finite dimensional $K$-algebra $B$,
\item
a maximal ideal $\mathfrak m$ of $B$ with $\pi:B\to K$ the quotient map,
\item
a $K$-basis $(\epsilon_0,\dots,\epsilon_\ell)$ for $B$ such that $\pi(\epsilon_0)=1$ and $\epsilon_1,\dots,\epsilon_\ell\in\mathfrak m$.
\end{itemize}

The following notion first appears, with somewhat different notation, in~\cite{paperA}.
It was inspired by Alexandru Buium's approach to differential algebra.

\begin{definition}
By a {\em $\mathcal D$-ring} we will mean a pair $(R,e)$ where $R$ is a $K$-algebra and $e:R\to R\otimes_KB$ is a $K$-algebra homomorphism satisfying $\pi^R\circ e=\id_R$.
Here $\pi^R=(\id_R\otimes_K\pi):R\otimes_KB\to R$ is the $R$-algebra homomorphism induced by~$\pi$.
We denote by $R^{\mathcal D}:=\{r\in R:e(r)=r\otimes 1\}$ the subring of {\em $\mathcal D$-constants}. 
\end{definition}

We will be applying Theorem~\ref{main-nr} to $X=\spec(R)$ when $(R,e)$ is a $\mathcal D$-ring and $R$ is a finitely generated $K$-algebra that is an integral domain.
We will set $Z=\spec(R\otimes_KB)$, $\phi_1:Z\to X$ the morphism induced by $e$, and $\phi_2:Z\to X$ the morphism induced by $r\mapsto r\otimes 1$.
Note that the nonreduced nature of $Z$ here is essential; $Z_{\red}=X$ and $\phi_1,\phi_2$ restricted to $Z_{\red}$ are both the identity.

But in order to see what the theorem will say in this context, we need to explore $\mathcal D$-rings a bit further.
First, two motivating examples.

\begin{example}[Differential rings]
\label{d-der}
Let $\mathcal D$ be given by the local $K$-algebra $K[\epsilon]/(\epsilon^2)$ with maximal ideal $\mathfrak m=(\epsilon)$ and $K$-basis $(1,\epsilon)$.
Suppose $R$ is a $K$-algebra equipped with a $K$-linear derivation $\delta:R\to R$.
Then we can make $R$ into a $\mathcal D$-ring by letting $e:R\to R[\epsilon]/\epsilon^2$ be $r\mapsto r+\delta(r)\epsilon$.
In fact, every $\mathcal D$-ring is of this form.
\end{example}

\begin{example}[Difference rings]
\label{d-diff}
Let $\mathcal D$ be given by the $K$-algebra $K\times K$ with maximal ideal $\mathfrak m$ generated by $(0,1)$ and $K$-basis $\big((0,1),(1,0)\big)$.
Then the $\mathcal D$-rings are precisely the $K$-algebras $R$ equipped with an endomorphism $\sigma:R\to R$, where $e:R\to R\times R$ is given by $r\mapsto (r,\sigma(r))$.
\end{example}

In fact, as suggested by the examples,  the $\mathcal D$-ring formalism is really a way to study rings equipped with certain operators.
Note that $(1\otimes\epsilon_0,\dots,1\otimes\epsilon_\ell)$ is an $R$-basis for $R\otimes_KB$, and $e:R\to R\otimes_KB$ can be written with respect to this basis so that for all $r\in R$,
$$e(r)=r\otimes\epsilon_0+\partial_1(r)\otimes\epsilon_1+\cdots+\partial_\ell(r)\otimes\epsilon_\ell$$
where $\partial_i:R\to R$ are $K$-linear operators on $R$.
(That the $\epsilon_0$-coefficient of $e(r)$ is~$r$ comes from the fact that $\pi^R\circ e=\id_R$ and $\pi(\epsilon_0)=1$.)
Writing $\partial:=(\partial_1,\dots,\partial_\ell)$, we can recover $e$ from $\partial$ and vice versa.
We will refer interchangeably to $(R,e)$ and $(R,\partial)$ as the $\mathcal D$-ring.

The class of operators $\partial$ that can be fit into this context is rather broad and robust, including various combinations and twists of differential and difference operators, and closed under various operations.
See paragraphs~3.3 through~3.7 of~\cite{paperC} for a discussion of examples.

Naturally associated to the operators $\partial$ on $R$ are certain $K$-algebra endomorphisms of $R$.
Let $\mathfrak m=\mathfrak m_0,\dots,\mathfrak m_t$ be the distinct maximal ideals of $B$, and $\pi=\pi_0,\pi_1,\dots,\pi_t$ the corresponding quotient maps $B\to K$.
Let $\sigma_i:=\pi_i^R\circ e:R\to R$ for $i=0,1,\dots, t$.
Note that $\sigma_0=\id_R$, and that $\sigma_1,\dots,\sigma_t$ are $K$-algebra endomorphisms of $R$ that are in fact $K$-linear combinations of the $\partial_1,\dots,\partial_\ell$.
We write $\sigma:=(\sigma_1,\dots,\sigma_t)$ and call $(R,\sigma)$ the {\em difference ring associated to $(R,e)$}.
Note that the assoicated endomorphism in the differential case of Example~\ref{d-der} is just the identity, and in the difference case of Example~\ref{d-diff} is $\sigma$ itself.

\begin{definition}[Totally invariant $\mathcal D$-ideals]
Suppose $(R,e)$ is a $\mathcal D$-ring.
An ideal $I\subseteq R$ is said to be a {\em $\mathcal D$-ideal} if $\partial_i(I)\subseteq I$ for all $i=1,\dots,\ell$, and {\em totally invariant} if $\sigma_j(I)=I$ for all $j=1,\dots,t$.
\end{definition}

In what follows, we view $R\otimes_KB$ as an $R$-algebra under the homomorphism $r\mapsto r\otimes 1$.
Hence, any ideal $I$ of $R$ generates an extension ideal of $R\otimes_KB$ which we denote by $I(R\otimes_KB)$.

\begin{proposition}
\label{d-ideal}
Suppose $(R,e)$ is a $\mathcal D$-ring with $R$ noetherian.
Let $I$ be an ideal of $R$.
Then $I$ is a totally invariant $\mathcal D$-ideal if and only if $e(I)(R\otimes_KB)=I(R\otimes_KB)$.
\end{proposition}

\begin{proof}
Note that $(1\otimes\epsilon_0,\dots,1\otimes\epsilon_\ell)$ is an $R$-basis for $R\otimes_KB$ and that $$I(R\otimes_KB)=\{a_0\otimes\epsilon_0+\cdots+a_\ell\otimes\epsilon_\ell:a_0,\dots,a_\ell\in I\}.$$
Suppose $e(I)(R\otimes_KB)=I(R\otimes_KB)$.
If $a\in I$ then
$$e(a)=a\otimes\epsilon_0+\partial_1(a)\otimes\epsilon_1+\cdots+\partial_\ell(a)\otimes\epsilon_\ell\in I(R\otimes_KB),$$
and hence $\partial_1(a),\cdots,\partial_\ell(a)\in I$.
So $I$ is a $\mathcal D$-ideal.
For total invariance, fixing $j=1,\dots,t$ and applying $\pi_j^R$ to $e(I)(R\otimes_KB)=I(R\otimes_KB)$ we get immediately that $\sigma_j(I)=I$.

Conversely, suppose $I$ is a totally invariant $\mathcal D$-ideal.
Then $e(a)\in I(R\otimes_KB)$ for all $a\in I$, since $\partial_1(a),\dots,\partial_\ell(a)\in I$.
That is, $e(I)(R\otimes_KB)\subseteq I(R\otimes_KB)$.

So it remains to show that $I(R\otimes_KB)\subseteq e(I)(R\otimes_KB)$ whenever $I$ is a totally invariant $\mathcal D$-ideal.

We first improve the choice of $K$-basis $(\epsilon_0,\dots,\epsilon_\ell)$.
Note that changing the basis, as long as $\pi(\epsilon_0)=1$ and $\epsilon_2,\dots,\epsilon_\ell\in\mathfrak m$ remain true, does not affect $e$ or $\sigma$ as these are intrinsically defined.
While it does change $\partial$ it does so only by replacing these operators with certain $K$-linear combinations of them.
In particular, the property of being a totally invariant $\mathcal D$-ideal is not affected.
We may therefore adjust the basis so that $\pi_i(\epsilon_j)=\delta_{i,j}$ for  $i,j=0,\dots,t$, and  $(\epsilon_{t+1},\dots,\epsilon_\ell)$ forms a $K$-basis for the Jacobson radical $J:=\bigcap_{j=0}^t\mathfrak m_j$.
Note that one of the consequences of this choice of basis is that  $\sigma_j=\partial_j$ for $j=0,\dots,t$.
(Recall that $\sigma_0=\partial_0=\id_R$.)

Suppose now that $I=(a_1,\dots,a_m)$.
For each $j=0,\dots,t$, since $\sigma_j(I)=I$, there is $a_{j,k}\in I$ such that $\sigma_j(a_{j,k})=a_k$ for all $k=1,\dots,m$.
Letting
$$y_k:=\sum_{j=0}^te(a_{j,k})(1\otimes\epsilon_j)$$
we have that for each $i=0,\dots,t$,
\begin{eqnarray*}
\pi^R_i(y_k)
&=&
\sum_{j=0}^t\pi_i^R\left(\sum_{n=0}^\ell\partial_n(a_{j,k})\otimes\epsilon_n\right)\pi_i^R(1\otimes\epsilon_j)\\
&=&
\sum_{j=0}^t\pi_i^R(\partial_i(a_{j,k})\otimes\epsilon_i)\delta_{i,j}\\
&=&
\sigma_i(a_{i,k})\\
&=&
a_k\\
&=&
\pi_i^R(a_k\otimes 1).
\end{eqnarray*}
The last equality uses the fact that $\pi_i^R$ is $R$-linear.
Hence $(a_k\otimes1)-y_k$ is of the form $b_{t+1}\otimes\epsilon_{t+1}+\cdots+b_\ell\otimes\epsilon_\ell$ for some $b_{t+1},\dots,b_\ell\in R$.
(Despite the notation, the $b_i$'s depend also on $k$.)
In fact, since $y_k\in e(I)(R\otimes_KB)\subseteq I(R\otimes_KB)$, we get that $b_{t+1},\dots,b_\ell\in I$.
Writing $b_\mu=\sum_{\nu=1}^mr_{\mu,\nu,k}a_\nu$, and setting $s_{\nu,k}:=-\sum_{\mu=t+1}^\ell(r_{\mu,\nu,k}\otimes\epsilon_\mu)$, we have that
$$y_k=(a_1\otimes 1)s_{1,k}+(a_2\otimes 1)s_{2,k}+\cdots+(a_k\otimes 1)(1+s_{k,k})+\cdots+(a_m\otimes 1)s_{m,k}$$
for all $k=1,\dots,m$.
This can be written in matrix notation as
$${\bf a}({\bf 1}+{\bf S})={\bf y}$$
where ${\bf a}=(a_1\otimes 1,\dots,a_m\otimes 1)$, ${\bf S}=(s_{\nu,k})\in \mat_m(R\otimes_KB)$, ${\bf 1}=\id_{\mat_m(R\otimes_KB)}$, and ${\bf y}=(y_1,\dots,y_m)$.
But since each $s_{\nu,k}\in R\otimes_KJ$, and $J$ is a nilpotent ideal of $B$, we get that ${\bf S}$ is nilpotent, and so ${\bf 1}+{\bf S}$ is invertible.
Hence,
$${\bf a}={\bf y}({\bf 1}+{\bf S})^{-1}\ \in\  \big(e(I)(R\otimes_KB)\big)^m.$$
That is, for each generator $a_k$ of $I$ in $R$ we have $a_k\otimes 1\in e(I)(R\otimes_KB)$.
Therefore $I(R\otimes_KB)\subseteq e(I)(R\otimes_KB)$, as desired.
\end{proof}

We are ready now to specialise Theorem~\ref{main-nr}.

\begin{definition}
By an {\em algebraic $\mathcal D$-variety} we mean an affine algebraic variety $X$ over $K$ whose co-ordinate ring $K[X]$ comes equipped with a $\mathcal D$-ring structure $e:K[X]\to B[X]$ whose associated endomorphisms of $K[X]$ are injective.
A Zariski closed subset of $X$ is said to be {\em totally $\mathcal D$-invariant} if its corresponding ideal is a totally invariant $\mathcal D$-ideal.
\end{definition}

\begin{remark}
The assumption that the associated endomorphisms are injective is to ensure that the $\mathcal D$-ring structure extends (uniquely) to the fraction field $K(X)$. See Lemma~4.9 of~\cite{paperC} for a proof of this.
\end{remark}

\begin{theorem}
\label{main-d}
Suppose $(X,e)$ is an algebraic $\mathcal D$-variety over $K$.
If $(X,e)$ has infinitely many totally $\mathcal D$-invariant hypersurfaces then there exists a $\mathcal D$-constant rational function $g\in K(X)\setminus K$.
\end{theorem}

\begin{proof}
Write $X=\spec(R)$ with $(R,e)$ a $\mathcal D$-ring.
Let $Z:=\spec\big(R\otimes_KB)$, $\phi_1:Z\to X$ the morphism induced by the $K$-algebra homomorphism $e:R\to R\otimes_KB$, and $\phi_2:Z\to X$ induced by $r\mapsto r\otimes 1$.
Note that $Z_{\operatorname{red}}=X$ and hence $Z$ is irreducible.
Moreover $\phi_1,\phi_2$ both restrict to the identity on $Z_{\operatorname{red}}$, and hence are dominant.
So Theorem~\ref{main-nr} applies.
By Proposition~\ref{d-ideal}, if $H\subseteq X$ is a totally $\mathcal D$-invariant hypersurface with ideal $I=I(H)$, then $e(I)(R\otimes_KB)=I(R\otimes_KB)$.
This means that $\phi_1^{-1}(H)=\phi_2^{-1}(H)$.
Hence, condition~(1) of Theorem~\ref{main-nr} holds with $U=X$ and $V=Z$.
The theorem gives us $g\in K(X)\setminus K$ with $g\phi_1=g\phi_2$.
That is, $e(g)=g\otimes 1$ under the canonical extension of $e$ to $\Frac(R)\to\Frac(R)\otimes_KB$.
We have found a nonconstant $\mathcal D$-constant rational function on $X$, as desired.
\end{proof}

When $\mathcal D$ is given by $B=K[\epsilon]/\epsilon^2$ as in Example~\ref{d-der} we recover the following consequence of a theorem of Jouanolou~\cite{jouanolou} on solutions to rational foliations: {\em an algebraic $D$-variety with infinitely many $D$-hypersurfaces admits a nonconstant rational first integral}.
This statement is the finite-dimensional case of Proposition~2.3 of Hrushovski's unpublished manuscript~\cite{hrushovski-jouanolou}; or, for a published proof, note that it is precisely the  $(m,r)=(1,0)$ case of~\cite[Theorem 4.2]{freitag-moosa}.
In fact, we get (a new proof of) the $r=0$ case of~\cite[Theorem 4.2]{freitag-moosa} for arbitrary $m\geq 1$ by applying  Theorem~\ref{main-d} to the case when $\mathcal D$ is given by $B=K[\epsilon_1,\dots,\epsilon_m]/(\epsilon_1,\dots,\epsilon_m)^2$.

\bigskip
\section{The reduced case and an application to rational dynamics}
\label{section-red}

\noindent
In this section we improve Theorem~\ref{main-nr} in the case when $Z$ is also a (reduced) algebraic variety, and thereby complete the proof of Theorem~\ref{main}.
We also deduce the application to rational dynamics discussed in the introduction.

First, for any function $\phi:Z\to X$ and any subset $H\subseteq X$, let us denote by $\phi^{-1}[H]$ the {\em set-theoretic inverse image} of the set $H$.
This is to avoid confusion with the notation $\phi^{-1}(H)$ we are using for the scheme-theoretic inverse image.
Now, suppose $\phi:Z\to X$ is a dominant rational map between algebraic varieties.\footnote{Recall that we view algebraic varieties as integral algebraic schemes, and so $\phi$ includes in particular a (partial) function on the underlying sets of these schemes.}
For a hypersurface $H\subseteq X$ with $H\cap\im(\phi)$ Zariski dense in $H$, by the {\em proper transform} of $H$, denoted by $\phi^{*}H$, we mean the hypersurface on $Z$ that is the union of those irreducible components of the Zariski closure of $\phi^{-1}[H]$ in $Z$ that project dominantly onto some irreducible component of $H$.

\begin{theorem}
\label{main-reduced}
Suppose $Z$ and $X$ are algebraic varieties and $\phi_1,\phi_2:Z\to X$ are dominant rational maps, over an algebraically closed field $K$ of characteristic zero.
Then the following are equivalent:
\begin{itemize}
\item[(1)]
There exist infinitely many hypersurfaces $H$ on $X$ satisfying
$\phi_1^{*}H=\phi_2^{*}H$.
\item[(2)]
There exists $g\in K(X)\setminus K$ such that $g\phi_1=g\phi_2$.
\end{itemize}
\end{theorem}

Note that when $Z=X$, $\phi_1=\phi$, and $\phi_2=\id$, this theorem says that if a rational dynamical system $(X,\phi)$ has infinitely many totally invariant hypersurfaces, then $\phi$ preserves a nonconstant rational function.
That is, we recover the algebraic case of~\cite[Corollary~3.3]{cantat}.
See also the closely related~\cite[Theorem~1.2]{BRS}.
But we can do better.
By a {\em rational finite self-correspondence} we will mean an algebraic variety $X$ equipped with a closed irreducible subvariety $\Gamma\subseteq X\times X$ such that the co-ordinate projections $\pi_1,\pi_2:\Gamma\to X$ are dominant and generically finite-to-one.
Note that we get a rational dynamical system by taking $\Gamma$ to be the graph of a dominant rational self-map.
A Zariski closed subset $V\subseteq X$ is {\em totally invariant} if its proper transforms in $\Gamma$ by the two co-ordinate projections agree.
A rational function $g$ on $X$ is {\em preserved by $\Gamma$} if $g\pi_1=g\pi_2$.

\begin{corollary}
\label{main-corr}
Suppose $(X,\Gamma)$ is a rational finite self-correspondence with infinitely many totally invariant hypersurfaces.
Then $\Gamma$ preserves a nonconstant rational function on $X$.
\end{corollary}

\begin{proof}
Apply Theorem~\ref{main-reduced} to $Z:=\Gamma$, $\phi_1:=\pi_1$, and $\phi_2:=\pi_2$.
\end{proof}

In fact, Theorem~\ref{main-reduced} is precisely the generalisation of the above corollary to arbitrary self-correspondences -- where the co-ordinate projections need not be generically finite-to-one.
As such, it can be viewed as a statement in generalised rational dynamics.

In order to deduce Theorem~\ref{main-reduced} from Theorem~\ref{main-nr} we need to observe that when working over a finitely generated field, and restricting attention to sufficiently small Zariski open sets, the scheme-theoretic inverse image and the proper transform agree on hypersurfaces.
This is Proposition~\ref{transform} below, and may very well be known, but we could not find it in the literature.
Our proof will rely on the following elementary, and certainly well-known, lemmas in commutative algebra.

\begin{lemma}
\label{radical-poly}
Suppose $A$ is a noetherian integral domain and $B=A[x_1,\dots,x_n]_g$ is the localisation of a polynomial algebra over $A$.
If $I\subseteq A$ is a radical ideal then so is~$IB$.
Moreover, if $A,B$ are in addition finitely generated $k$-algebras for some field $k$, $\phi:\spec(B)\to\spec(A)$ is the induced morphism of $k$-varieties, $V:=V(I)\subseteq\spec A$ is the corresponding subvariety, and $g\notin IA[x_1,\dots,x_n]$, then $\phi^{-1}(V)=\phi^{*}V$.
\end{lemma}

\begin{proof}
It is straightforward to check that localisation preserves radicality.
That taking polynomial extensions also preserve radicality follows from:
\begin{itemize}
\item[(a)]
if $P\subset A$ is a prime ideal then so is $PA[x_1,\dots,x_n]$, and
\item[(b)]
for prime ideals $P_1,\dots,P_\ell$ of $A$,
$$\big(\bigcap_{i=1}^\ell P_i\big)A[x_1,\dots,x_n]=\bigcap_{i=1}^\ell\big(P_iA[x_1,\dots,x_n]\big).$$
\end{itemize}
The ``moreover" clause follows by first noting that since $IB$ is radical, the scheme-theoretic and set-theoretic inverse images of $V=V(I)$ agree.
Moreover, 
if $P$ is a minimal prime ideal of $A$ containing $I$ then, by~(a) and the fact that $g\notin PA[x_1,\dots,x_n]$, $PB$ is a prime ideal.
That is, the irreducible components of $\phi^{-1}[V]$ are of the form $\phi^{-1}[W]$ where $W$ is an irreducible component of $V$.
Hence the proper transform agrees with the set-theoretic inverse image of $V$.
\end{proof}

\begin{lemma}
\label{radical}
If $A\subseteq B$ is an \'etale extension of noetherian unique factorisation domains, and $I$ is a height one radical ideal of $A$, then $IB$ is radical.
Moreover, if $A,B$ are in addition finitely generated $k$-algebras for some field $k$, $\phi:\spec(B)\to\spec(A)$ is the induced morphism of $k$-varieties, and $H:=V(I)$ is the corresponding hypersurface on $\spec(A)$, then $\phi^{-1}(H)=\phi^{*}H$.
\end{lemma}

\begin{proof}
Let $P_1,\dots, P_\ell$ be the distinct minimal prime ideals containing $I$ in~$A$.
Since $A$ is a UFD and each $P_i$ is of height one, we have that $P_i=p_iA$ for some irreducible $p_i\in A$.
Since $I=P_1\cap\cdots\cap P_\ell$ and the $p_i$ are mutually non-associate, we get that $I=(p_1\cdots p_\ell)A$.
Let $p_i=u_iq_{i,1}^{n_{i,1}}\cdots q_{i,m_i}^{n_{i, m_i}}$ be the prime factorisation of $p_i$ in the UFD $B$.
Now, each $Q_{i,j}:=q_{i,j}B$ is a minimal prime ideal of $P_iB$, and hence by the going down theorem for flat extensions, $Q_{i,j}$ lies over $P_i$.
In particular, the $q_{i,j}$ are non-associate even as $i$ varies.
But moreover, as $B$ over $A$ is unramified, $P_iB_{Q_{i,j}}=Q_{i,j}B_{Q_{i,j}}$.
That is,  $q_{i,j}^{n_{i,j}}B_{Q_{i,j}}=q_{i,j}B_{Q_{i,j}}$.
This forces each $n_{i,j}=1$.
So $IB=(\prod_{i=1}^\ell\prod_{j=1}^{m_i}q_{i,j})B$ is radical.

For the ``moreover" clause, again we first observe that the set-theoretic and scheme-theoretic inverse images of $H=V(I)$ agree because $IB$ is radical.
Now, the irreducible components of $\phi^{-1}[H]=V(IB)$ are the $V(q_{i,j})$.
That $Q_{i,j}$ lies over $P_i$ says exactly that $V(q_{i,j})$ projects dominantly onto $V(p_i)$, which is an irreducible component of $H$.
Hence $\phi^{-1}[H]=f^*H$, as desired.
\end{proof}

\begin{proposition}
\label{transform}
Suppose $\phi:Z\to X$ is a dominant rational map between algebraic varieties over a finitely generated field $k$.
There exist nonempty Zariski open subsets $V\subseteq Z$ and $U\subseteq X$ such that the restriction $\phi^V:V\to U$ is a dominant regular morphism, and for all but finitely many hypersurfaces $H$ on $U$, $(\phi^V)^{-1}(H)=(\phi^V)^{*}H$.
\end{proposition}

\begin{proof}
Replacing $Z$ and $X$ by nonempty Zariski open subsets, it suffices to prove the proposition in the case when $X=\spec(R)$ and $Y=\spec(S)$ are affine $k$-varieties and $\phi$ is a dominant $k$-morphism induced by an injective $k$-algebra homomorphism $f:R\to S$.
Now, as we have used before, that $k$ is a finitely generated field implies that the localisation of $R$ (respectively $S$) at some nonzero element is a unique factorisation domain -- this is~\cite[Lemma~6.11]{bllsm}.
So we may assume that $R$ and $S$ are already unique factorisation domains.

Next, by Noether's normalisation lemma, after replacing $S$ with $S_g$ for some nonzero $g$, we may assume that the homomorphism $f$ factors through injective $k$-algebra homomorphisms $R\to R'$ and $R'\to S$ where $R'$ is a polynomial algebra over $R$ and $S$ is quasi-finite over $R'$.
Localising both $R'$ and $S$ further, we may in fact take $R'\to S$ to be \'etale, though now $R'$ is a finitely generated localisation of a polynomial algebra over $R$.

So we have that $\phi$ factors as $\spec(S)\to \spec(R')\to\spec(R)$.
Since $R'$ is of the form $R[x_1,\dots,x_n]_g$, Lemma~\ref{radical-poly} tells us that if $I=I(H)$ is a radical height one ideal in $R$ then $IR'$ is radical.
Moreover, since $V(g)$ can only contain finitely many hypersurfaces, for all but finitely many such $I$, $IR'$ is again of height one.
Since $S$ is \'etale over~$R'$, Lemma~\ref{radical} now applies and we get that $(IR') S=f(I)S$ is radical.
The ``moreover" clauses in the lemmas tell us that $\phi^{-1}(H)=\phi^*H$.
\end{proof}

\begin{proof}[Proof of Theorem~\ref{main-reduced}]
That~(2) implies~(1) is again clear: the level sets of $g$ will witness~(1).
Suppose~(1) holds.
Exactly as in the beginning of the proof of Theorem~\ref{main-nr} we can find a finitely generated subfield $k\subseteq K$ over which $Z,X,\phi_1,\phi_2$ are defined and such that there is an infinite set $\mathcal H$ of hypersurfaces $H$ on $X$ over $k$ satisfying
$\phi_1^{*}H=\phi_2^{*}H$.
Applying Proposition~\ref{transform} to $(\phi_\nu)_k:Z_k\to X_k$, there exist nonempty Zariski open subsets $V\subseteq Z_k$ and $U\subseteq X_k$ such that, for $\nu=1,2$, the restrictions $(\phi_\nu)_k^{V}:V\to U$ are dominant regular morphisms and for all but finitely many $H\in\mathcal H$,
$((\phi_\nu)_k^V)^{-1}(H_k\cap U)=((\phi_\nu)_k^V)^{*}(H_k\cap U)$.
Noting that proper transforms commute with extending the base field, and observing that $U$ and $V$ avoid only finitely many hypersurfaces on $X$ and $Z$ respectively, we have that for all but finitely many $H\in\mathcal H$, 
$((\phi_1)_k^V)^{*}(H_k\cap U)=((\phi_2)_k^V)^{*}(H_k\cap U)$.
Hence, $((\phi_1)_k^V)^{-1}(H\cap U)=((\phi_2)_k^V)^{-1}(H\cap U)$ for all but finitely many $H\in\mathcal H$.
As scheme-theoretic inverse images also commute with extending the base field, we get that
$(\phi_1^{V_K})^{-1}(H\cap U_K)=(\phi_2^{V_K})^{-1}(H\cap U_K)$.
This witnesses the truth of condition~(1) of Theorem~\ref{main-nr}, and so by that theorem, condition~(2) holds.
\end{proof}

\bigskip
\section{Normal varieties equipped with prime divisors}
\label{section-terminal}

\noindent
Theorem~\ref{main-reduced} is really about the birational geometry of algebraic varieties equipped with a set of hypersurfaces.
We will show how a direct study of this category leads us to an alternative, more geometric and conceptual, proof of Theorem~\ref{main-reduced} in the case when we assume that the fibres of $\phi_1$ and $\phi_2$ are irreducible.
This section is self-contained and largely independent from the rest of the paper.

Fix a field $k$ of characteristic $0$, and let $K:=k^{\alg}$.
We consider the following category $\mathcal V_k$.
The {\em objects} of $\mathcal V_k$ are pairs $(X,S)$ where $X$ is a normal geometrically irreducible algebraic variety over $k$ and $S$ is a set of prime divisors (i.e., irreducible hypersurfaces) on $X_K:=X\times_kK$.
A {\em morphism} $(X,S)\to (Y,T)$ will be a dominant rational map $\phi:X\to Y$ over $k$ whose generic fibre is geometrically irreducible, and such that the symmetric difference between $S$ and $\{\phi_K^*t:t\in T\}$ is finite.
Note that this implies, in particular, that at most finitely many of the prime divisors in $S$ can project dominantly onto $Y$.

\begin{remark}
\label{ptirred}
Here $\phi_K:X_K\to Y_K$ is the base extension of $\phi:X\to Y$ to $K$.
Because the generic fibre of $\phi_K$ is irreducible, the proper transform of all but finitely many prime divisors on $Y_K$ is a prime divisor on $X_K$.
Indeed, if $t$ is an irreducible hypersurface on $Y_K$ that has nonempty intersection with the Zariski open subset of points in $Y_K$ over which the fibre of $\phi_K$ is irreducible, then $\phi_K^*t$ will be irreducible.
\end{remark}

Note that in this category the underlying varieties and rational maps are over~$k$ but the irreducible hypersurfaces they come with may be over the algebraic closure~$K$.
Things would become notationally much clearer if we assumed that $k$ is algebraically closed, but in fact the main theorem will only apply when $k$ is a finitely generated field.
We will systematically use the subscript $_K$ to indicate base extension from $k$ up to $K$.
One exception, however, will be for fields of rational functions: For $X$ a geometrically irreducible algebraic variety over $k$ we will denote by $K(X)$ the field of rational functions on $X_K$.

The category of algebraic varieties over $k$ has a terminal object, namely $\spec(k)$.
At first sight one might think that $(\spec(k),\emptyset)$ is the terminal object in $\mathcal V_k$, but this is not the case.
If $S$ is a {\em finite} set then there is a canonical morphism $(X,S)\to (\spec(k),\emptyset)$, but if $S$ is infinite then it is not hard to see that the existence of a morphism $(X,S)\to(Y,T)$ implies $\dim Y>0$.
We seek to repair this lack of terminal object by asking if the {\em undercategory} of arrows originating at a given $(X,S)$ in $\mathcal V_k$ has a terminal object.

\begin{theorem}
\label{terminal}
Suppose $k$ is finitely generated.
For every object $(X,S)$ in $\mathcal V_k$ there is a morphism $\pi:(X,S)\to(X',S')$ that is terminal with respect to all morphisms originating from $(X,S)$.
That is, given $\phi:(X,S)\to (Y,T)$ there is a unique $\psi:(Y,T)\to(X',S')$ such that $\psi\phi=\pi$.
\end{theorem}

\begin{proof}
First of all, we can embed $X$ as an open subvariety of a normal variety $\overline X$ which is proper over $k$.
Let $\overline S$ denote the set of Zariski closures of elements of $S$ in $\overline X_K$.
The embedding of $X$ in $\overline X$ induces an isomorphism $(X,S)\cong(\overline X,\overline S)$ in $\mathcal V_k$.
It suffices therefore to prove the theorem for $(\overline X,\overline S)$.
That is, we may assume $X$ is proper over $k$.

Our assumption of normality means that for any rational function $f\in K(X)$ we can consider the Weil divisor $\divisor(f)$ on $X_K$.
By the {\em support} of $f$ we mean the set of prime divisors appearing in $\divisor(f)$ with nonzero coefficient -- so it is the set of ``zeros" and ``poles" of $f$.
Given a set $T$ of prime divisors on $X_K$, let us denote by $T^\sharp\subseteq K(X)$ the set of rational functions whose support is contained in $T$, and by $K_T$ the relative algebraic closure of $K(T^\sharp)$ in $K(X)$.

Consider the natural action of $\gal(k)$ on $X_K$ coming from the fact that $X$ is over~$k$.
For any set $T$ of prime divisors on $X_K$, let $\widehat T$ denote the set of all $\gal(k)$-conjugates of elements  of $T$.

We claim that there is a cofinite subset $S_0 \subseteq S$ such that for all cofinite $T \subseteq S_0$, one has $K_{\widehat{T}} = K_{\widehat{S_0}}$.
If this were not true, we would be able to construct a strictly descending infinite chain of cofinite subsets of $S$,
\[ \cdots \subsetneq S_2 \subsetneq S_1 \subsetneq S, \]
such that
\[ \cdots \subsetneq K_{\widehat{S_2}} \subsetneq K_{\widehat{S_1}} \subsetneq K_{\widehat{S}}. \]
Since $K_{\widehat{T}}$ is a relatively algebraically closed subextension of $K(X)$ over $K$ by definition, it would follow that
\[ 0 \leq \cdots < \operatorname{trdeg}(K_{\widehat{S_2}}/K) < \operatorname{trdeg}(K_{\widehat{S_1}}/K) < \operatorname{trdeg}(K_{\widehat{S}}/K)\leq \operatorname{trdeg}(K(X)/K). \]
But $K(X)$ has finite transcendence degree over $K$, so this is impossible.

For such a cofinite $S_0 \subseteq S$, the identity map is an isomorphism between $(X,S)$ and $(X,S_0)$, so it suffices to prove the theorem for $(X,S_0)$.
In other words, by replacing $S$ with such a sufficiently small cofinite subset, we may also assume that $K_{\widehat{T}} = K_{\widehat{S}}$ for all cofinite subsets $T \subseteq S$.

There is also a natural action of $\gal(k)$ on $K(X)$.
As $\widehat S$ is $\gal(k)$-invariant, so is the set of rational functions~$\widehat S^\sharp$, and hence also the subfield $K_{\widehat S}\subseteq K(X)$.
This means that $K_{\widehat S}$ is the function field of a $K$-variety that descends to $k$, that is, $K_{\widehat S}=K(X')$ for some normal geometrically irreducible algebraic variety $X'$ over $k$, and the embedding $K(X')\subseteq K(X)$ comes from a dominant rational map $\pi:X\to X'$ over $k$.
As $K_{\widehat S}$ is relatively algebraically closed in $K(X)$ the generic fibre of $\pi$ is also geometrically irreducible.

We claim that only finitely many $s\in S$ map dominantly onto $X'_K$ by $\pi_K$.
Suppose towards a contradiction that infinitely many elements of $S$ map dominantly onto $X_K'$.
By the Mordell-Weil-N\'eron-Severi theorem (see \cite[Corollary 6.6.2]{Lang} for details) the divisor class group $\operatorname{Cl}(X)$ is finitely generated (as $k$ is a finitely generated field).
Let $n$ be bigger than the rank of $\operatorname{Cl}(X)$.
Choose $s_1,\dots,s_n\in S$ that map dominantly onto $X'_K$ and have distinct $\gal(k)$-orbits.
Note that as $\pi$ is over~$k$ the $\gal(k)$-conjugates of the $s_i$ also map dominantly onto $X'_K$.
If we let $H_i$ be the union of the $\gal(k)$-conjugates of $s_i$, then $H_i$ descends to $k$ and is $k$-irreducible.
That is, we have distinct prime divisors $d_1,\dots,d_n$ on $X$ over~$k$ such that $H_i={d_i}_K$.
By choice of $n$ there are rational (and so integer) numbers $r_1,\dots,r_n$ not all zero, and $f\in k(X)\setminus k$, such that $\sum r_i d_i=\divisor(f)$ in $\operatorname{Div}(X)$.
So, working again over $K$, the support of $f\in K(X)$ is contained in $\{s_i^\sigma:i=1,\dots,n, \sigma\in\gal(k)\}$.
In particular, $f\in K_{\widehat S}=K(X')$.
But the pullback of a rational function on $X'_K$ cannot have prime divisors in its support that project dominantly onto $X'_K$ -- they must all project onto prime divisors on $X'_K$.
This contradiction proves that only finitely many $s\in S$ map dominantly onto $X'_K$.

There are also only finitely many $s\in S$ that live in the indeterminacy locus of $\pi_K$ as that indeterminacy locus is of codimension $\geq 1$.
Let $S_0$, therefore, be the cofinitely many $s\in S$ that are neither in the indeterminacy locus of $\pi_K$ nor do they project dominantly onto $X'_K$ by $\pi_K$.
For each $s \in S_0$, the Zariski closure of $\pi_K(s)$ is then a proper irreducible subvariety of $X'_K$, which we will denote by $s'$.
By dimension considerations, for cofinitely many $s \in S_0$, the corresponding $s'$ is a prime divisor on $X'_K$.
By Remark 7.1, $\pi_K^*s' = s$ for all but finitely many of these~$s$.
Let $S_1$ be the cofinite subset of $s\in S_0$ such that if $s'$ is the Zariski closure of $\pi_K(s)$, then $s'$ is a prime divisor on $X'_K$ and $\pi_K^*s' = s$.
Let
$$S':=\{\text{Zariski closure of }\pi_K(s):s\in S_1\}.$$
We have that $\pi:(X,S)\to(X',S')$ is a morphism in $\mathcal V_k$.

It remains to show $\pi$ is terminal.
Given a morphism $\phi:(X,S)\to (Y,T)$, we seek to complete the triangle
$$\xymatrix{
(X,S)\ar[r]^{\phi}\ar[d]_{\pi}&(Y,T)\\
(X',S')
}$$
with a morphism $\psi:(Y,T)\to (X',S)$.

Since $\phi:(X,S)\to(Y,T)$ is a morphism, it must be that only finitely many $s\in S$ map dominantly onto $Y_K$ under $\phi_K$.
Replacing $S$ by a cofinite subset, we may assume that there are no such $s\in S$.
It follows from the fact that $\phi:X\to Y$ is over $k$ that also no elements of $\widehat S$ will map dominantly onto $Y_K$.
Now, supppose $f\in\widehat S^\sharp$.
Then no member of the support of $f$ maps dominantly onto $Y_K$.
This means that $f$ has no zeros or poles on the generic fibre $X_\eta$ of $\phi_K$.
By the properness of~$X$, and hence of $X_\eta$ over $K(Y)$, we must have that $f$ is constant on the generic fibre.
So $f$ is the pull-back of a rational function on $Y$.
That is, $\widehat S^\sharp\subseteq K(Y)$, and hence $K(X')=K_{\widehat S}\subseteq K(Y)$.
We thus obtain a dominant rational map $\alpha:Y_K\to X'_K$ with irreducible generic fibre such that $\pi_K=\alpha\phi_K$.
Since $\phi$ is dominant there is a unique such $\alpha$.
Since $\pi$ and $\phi$ are over $k$, an automorphism argument shows that $\alpha$ descends to $k$, that is, $\alpha=\psi_K$ for some dominant rational map $\psi:Y\to X'$ with geometrically irreducible generic fibre.
And we have
$$\xymatrix{
X\ar[r]^{\phi}\ar[d]_{\pi}&Y\ar[dl]^{\psi}\\
X'
}$$
It remains to verify that the symmetric difference between $T$ and $\{\psi_K^*s':s'\in S'\}$ is finite.
But a diagram chase shows that for cofinitely many $s'\in S'$, $\psi_K^*s'$ is the Zariski closure of $\phi_K(\pi_K^*s')$.
Hence, for cofinitely many $s'\in S'$, $\psi_K^*s'\in T$.
On the other hand,
for cofinitely many $t\in T$, $t$ is the proper transform under $\psi$ of the Zariski closure of $\pi_K(\phi_K^*t)$ in $X'_K$, which is in $S'$ for cofinitely many $t$.
\end{proof}

While we think the above theorem may be of independent interest, our immediate motivation is the following alternative proof of a special case of Theorem~\ref{main-reduced}.

\begin{corollary}
\label{main-reduced-irred}
Suppose $Z$ and $X$ are algebraic varieties and $\phi_1,\phi_2:Z\to X$ are dominant rational maps with generic fibres irreducible, all over an algebraically closed field $K$ of characteristic zero.
If there exist infinitely many hypersurfaces $H$ on $X$ satisfying
$\phi_1^{*}H=\phi_2^{*}H$ then there is $g\in K(X)\setminus K$ with $g\phi_1=g\phi_2$.
\end{corollary}

\begin{proof}
The general idea of proof is to use Theorem~\ref{terminal} to reduce to the case of a rational dynamical system, and then apply the results of Cantat in that setting (namely Corollary~\ref{cor-cantat} of the introduction).

First we reduce to the case that $\dim X<\dim Z$.
Indeed, suppose $\dim X=\dim Z$.
Then, both $\phi_1$ and $\phi_2$ are birational, and we can consider the birational self-map $\alpha:=\phi_2\phi_1^{-1}:X\to X$.
If $H$ is a hypersurface on $X$ for which $\phi_1^*H=\phi_2^*H$ then $\alpha^*H=H$.
We thus have infinitely many totally invariant hypersurfaces on $X$ for the rational dynamical system $(X,\alpha)$.
By Corollary~\ref{cor-cantat} there is a nonconstant $g\in K(X)$ such that $g \alpha=g$.
Precomposing with $\phi_1$ yields $g\phi_2=g\phi_1$, as desired.
We may therefore assume that $\dim X<\dim Z$.

Suppose $\mathcal H$ is a countably infinite set of hypersurfaces on $X$ whose strict transforms with respect to $\phi_1$ and $\phi_2$ agree.
One complication is that the $H\in\mathcal H$ are not necessarily irreducible, and to deal with that we argue now that we may assume that no two members of $\mathcal H$ share an irreducible component in common.
First some notation:
for $H$ a hypersurface on $X$, let $S_H$ denote the (finite) set of its irreducible components.
Note that $\phi_1^*H=\phi_2^*H$ if and only if $\phi_1^*(S_H)=\phi_2^*(S_H)$ as sets of prime divisors on $Z$.
Now, enumerate $\mathcal H=\{H_0,H_1,\dots\}$ and define a new sequence $H_i'$ recursively by setting $H_0':=H_0$ and $H_{i+1}'$ to be the union of the prime divisors in the set $S_{H_{i+1}}\setminus(\bigcup_{j=0}^iS_{H_j'})$.
Then we get a sequence $H_0',H_1',\dots$ whose nonempty members are hypersurfaces on $Y$ that still satisfy $\phi_1^*H_i'=\phi_2^*H_i'$, because $\phi_1^*(S_{H_i'})=\phi_2^*(S_{H_i'})$.
No two nonempty members of this sequence share an irreducible component.
Moreover, there are infinitely many nonempty $H_i'$, as at any finite stage $\bigcup_{j=0}^iS_{H_j'}$ is a finite set of irreducible hypersurfaces.
So we may as well assume that distinct members of the original $\mathcal H$ share no irreducible components.

We now proceed by induction on the dimension of $Z$, with $\dim Z=0$ being vacuous.
For each $H\in\mathcal H$, let $S_H$ denote the (finite) set of irreducible components of $H$, and $T_H$ the set of irreducible components of $\phi_1^*H=\phi_2^*H$ in $Z$.
Set $\displaystyle S:=\bigcup_{H\in\mathcal H}S_H$ and $\displaystyle T:=\bigcup_{H\in\mathcal H}S_H$.
Let $k_0$ be a finitely generated subfield over which $Z, X, \phi_1, \phi_2$ are defined.
As the statement of the corollary is preserved under birational equivalence, we may assume that $Z$ and $X$ are normal.
Suppose for now that all the members of $\mathcal H$ (and hence $S$ and $T$) are defined over $K_0:=k_0^{\alg}$.
So $(Z,T),(X,S)$ are objects in $\mathcal V_{k_0}$, and $\phi_1,\phi_2:(Z,T)\to (X,S)$ are morphisms in~$\mathcal V_{k_0}$.
By Theorem~\ref{terminal} there is a terminal morphism $\pi:(Z,T)\to (Z',T')$.
We get an induced diagram in $\mathcal V_{k_0}$,
$$\xymatrix{
(Z,T)\ar[drr]^{\pi}\ar[d]_{\phi_2}\ar[rr]^{\phi_1}&&(X,S)\ar[d]^{\psi_1}\\
(X,S)\ar[rr]^{\psi_2}&&(Z',T').
}$$
We want to apply the induction hypothesis to $\psi_1,\psi_2:X\to Z'$.
To do so, we first remove any $H$ from $\mathcal H$ that maps dominantly onto $Z'$ by $\psi_1$.
There can only be finitely many elements of $S$ with this property since $\psi_1$ is a  morphism in
$\mathcal V_{k_0}$.
As $S$ is the set of prime divisors appearing as components of elements of $\mathcal H$, and as no two members of $\mathcal H$ share an irreducible component, there are only finitely many $H$'s to remove.
Removing finitely many more, we may assume that for all $H\in\mathcal H$ the Zariski closure of $\psi_1(H)$, which we denote by $H'$, is a hypersurface on $Z'$, and that $\psi_1^*H'=H$.
Chasing the above diagram, we get that $\psi_2^*H'=H$ also.
We see that there are infinitely many hypersurfaces $H'$ on $Z'$ such that $\psi_1^*H'=\psi_2^*H'$.
By the inductive hypothesis (as $\dim X<\dim Z$) there exists nonconstant $g'\in K(Z')$ such that $g'\psi_1=g'\psi_2=:g\in K(X)$.
So $g\phi_1=g\phi_2$, as desired.

We still have to consider that case that there is an $H\in \mathcal H$ that is not defined over $K_0$.
But we have seen how to deal with this before:
$H$ then is defined over a finitely generated {\em nonalgebraic} extension $k_1$ of~$k_0$.
Let $K_1:=k_1^{\alg}$.
There are infinitely many $\operatorname{Aut}(K_1/k_0)$-conjugates of $H$ and they all satisfy the property that their strict transforms under $\phi_1$ and $\phi_2$ agree, because $\phi_1$ and $\phi_2$ are defined over $k_0$.
Letting $\widetilde{\mathcal H}$ be this infinite set, and working in $\mathcal V_{k_1}$ with  $\widetilde{\mathcal H}$, rather than in $\mathcal V_{k_0}$ with  $\mathcal H$, we can carry out the above argument.
\end{proof}

\bigskip
\section{Positive characteristic}
\label{section-charp}

\noindent
We have worked so far exclusively in characteristic zero, mostly because the differential algebraic techniques we employ in dealing with the nonreduced case very much require it.
But it is reasonable to ask to what extent our proof of the reduced case  can be extended to positive characteristic.

The first thing to observe is that even the special case of Cantat's theorem (Corollary~\ref{cor-cantat}) is false in positive characteristic: consider the dynamical system $(\mathbb P_1,\Fr_p)$  on the projective line over the prime finite field $\mathbb F_p$ equipped with the $p$-power Frobenius morphism;
there are no preserved nonconstant rational functions, but the $\gal(\mathbb F_p)$-orbit of any point in $\mathbb P_1(\mathbb F_p^{\alg})$ is a totally invariant hypersurface.
Our proof breaks down in Proposition~\ref{transform} where we replaced scheme-theoretic inverse images by proper transforms; we used the characteristic zero fact that, after localising, a quasi-finite extension can be made \'etale.
The natural way to deal with this would be to impose some separability condition: we should ask that the dominant rational maps $\phi_1,\phi_2:Z\to X$ have generic fibres that are {\em geometrically reduced}, or what is equivalent, that the function field extensions they induce admit separating transcendence basis.
This is of course automatic in characteristic zero, and in positive characteristic rules out the Frobenius example.
Indeed, the proof of Proposition~\ref{transform} simply goes through in arbitrary characteristic with this additional assumption.\footnote{Proposition~\ref{transform} does make use of~\cite[Lemma~6.11]{bllsm} which is stated for characteristic zero.
However, the proof given there goes through in positive characteristic if we replace the use of Mordell-Weil with Lang-N\'eron.}

However, there is another key point in the proof of Theorem~\ref{main-reduced} where characteristic zero is used.
In reducing to the case when infinitely many of the invariant hypersurfaces are defined over the same finitely generated field~$k$, we first get them over $k^{\alg}$ and then take the union of the Galois-conjugates.
In positive characteristic these hypersurfaces will now only be guaranteed to be over the perfect hull of~$k$, which is not necessarily finitely generated.
We do not see how to avoid this problem and are thus left with the following partial result in arbitrary characteristic.

\begin{theorem}
\label{main-charp}
Fix $K$ an algebraically closed field of arbitrary characteristic.
Suppose $\phi_1,\phi_2:Z\to X$ are dominant rational maps between algebraic varieties over~$K$ with geometrically reduced generic fibres.
Then the following are equivalent:
\begin{itemize}
\item[(1)]
There exists a finitely generated subfield $k\subseteq K$ and infinitely many hypersurfaces $H$ on $X$ defined over $k^{\sep}$ satisfying
$\phi_1^{*}H=\phi_2^{*}H$.
\item[(2)]
There exists $g\in K(X)\setminus K$ such that $g\phi_1=g\phi_2$.
\end{itemize}
\end{theorem}

\begin{proof}
This is obtained by inspecting the proofs in characteristic zero, together with the preceding remarks. We give only a brief sketch.

For~(2)$\implies$(1), let $k$ be a finitely generated field over which $Z,X,\phi_1,\phi_2, g$ are defined.
Then the level sets of $g$ over $k^{\sep}$ give rise to infinitely many hypersurfaces satisfying $\phi_1^{*}H=\phi_2^{*}H$.

Suppose~(1) holds.
We may assume that $Z,X,\phi_1,\phi_2$ are all defined over $k$ as well.
Replacing the hypersurfaces by the union of their $\gal(k)$-conjugates, we may assume that they are all defined over $k$ itself.
As discussed above, because of our assumption of geometrically reduced generic fibres, Proposition~\ref{transform} remains true.
Hence, exactly as in the proof of Theorem~\ref{main-reduced}, after replacing $Z$ and $X$ with sufficiently small nonempty Zariski open subsets, we may assume that we have an infinite sequence $(H_j:j<\omega)$ of hypersurfaces satisfying $\phi_1^{-1}(H_j)=\phi_2^{-1}(H_j)$.
We now follow the proof of Theorem~\ref{main-nr} keeping in mind that $Z$ is reduced, but that the characteristic need not be zero.
Possibly shrinking $X$ and $Y$ further, we may assume $X=\spec(R_K)$ and $Y=\spec(S_K)$ where $R$ and $S$ are finitely generated $k$-algebras,  $R$ is a UFD, $S$ is an integral domain, $k$ is relatively algebraically closed in $\Frac(R)$ and $\Frac(S)$, and $\phi_1,\phi_2$ are induced by $k$-algebra embeddings $f_1,f_2:R\to S$.
The hypersurfaces $(H_j:j<\omega)$ must have principal vanishing ideals and so we get a sequence $(a_j:j<\omega)$ in $R$ that is multiplicatively independent modulo $k^\times$ and, because the $H_j$ satisfy $\phi_1^{-1}(H_j)=\phi_2^{-1}(H_j)$,  the $a_j$ satisfy $f_1(a_j)S=f_2(a_j)S$.
The hypothesis of Theorem~\ref{main-algebra} are satisfied, except that we may be in positive characteristic.
But the proof of Theorem~\ref{main-algebra} in the case when $S$ is an integral domain -- this is the first three paragraphs of that proof -- did not use characteristic zero.
Hence, there exists $g\in\Frac(R)\setminus k$ such that $f_1(g)=f_2(g)$.
This proves~(2).
\end{proof}

\begin{question}
\label{dropfg}
Is the assumption in~\ref{main-charp}(1) of the existence of a common finitely generated field of definition necessary?
\end{question}

It may be worth pointing out that Theorem~\ref{terminal} on the category of normal varieties equipped with a set of prime divisors remains true in positive characteristic up to applications of Frobenius transforms -- but this does not seem to help in answering Question~\ref{dropfg} even when the generic fibres of $\phi_1,\phi_2$ are assumed to be irreducible.
\vfill
\pagebreak

%\bibliographystyle{plain}
%\bibliography{../principal}

\end{document}